\documentclass[a4paper,11pt]{amsart}
\usepackage[ascii]{inputenc}
\usepackage{amsmath,amsfonts,amssymb,amsthm}
\usepackage{mathrsfs}     
\usepackage{xcolor} 
\usepackage{a4wide}
\usepackage{verbatim}     
             
\newtheoremstyle{theorem}{5pt}{5pt}{\itshape}{}{\bfseries}{.}{.5em}{}
     
\theoremstyle{theorem}  
\newtheorem{theorem}{Theorem}           
\newtheorem{lemma}[theorem]{Lemma}
\newtheorem{corollary}[theorem]{Corollary}

\newtheorem{remark}[theorem]{Remark}

\usepackage{titlesec}
\titleformat{\section}
{\normalfont\bfseries}{\large\thesection}{1em}{\large}
\titlespacing*{\section}{0pt}{3.5ex plus 1ex minus .2ex}{2.3ex plus .2ex}
\titleformat{\subsection}
{\normalfont\bfseries}{\thesubsection}{1em}{\normalsize}
\titlespacing*{\subsection}{0pt}{3.5ex plus 1ex minus .2ex}{2.3ex plus .2ex}

\makeatletter
\def\Ddots{\mathinner{\mkern1mu\raise\p@
\vbox{\kern7\p@\hbox{.}}\mkern2mu
\raise4\p@\hbox{.}\mkern2mu\raise7\p@\hbox{.}\mkern1mu}}
\makeatother

\allowdisplaybreaks[3]
                         
\begin{document}  

\title{On Signs of Fourier Coefficients of Hecke-Maass Cusp Forms on $\mathrm{GL}_3$}
\author{Jesse J\"a\"asaari}
\address[]{School of Mathematical Sciences, Queen Mary University of London, London E1 4NS, UK}
\email{j.jaasaari@qmul.ac.uk}
\date{}
\maketitle

\begin{abstract}
We consider sign changes of Fourier coefficients of Hecke-Maass cusp forms for the group $\mathrm{SL}_3(\mathbb Z)$. When the underlying form is self-dual, we show that there are $\gg_\varepsilon X^{5/6-\varepsilon}$ sign changes among the coefficients $\{A(m,1)\}_{m\leq X}$ and that there is a positive proportion of sign changes for many self-dual forms. Similar result concerning the positive proportion of sign changes also hold for the real-valued coefficients $A(m,m)$ for generic $\mathrm{GL}_3$ cusp forms, a result which is based on a new effective Sato-Tate type theorem for a family of $\mathrm{GL}_3$ cusp forms we establish. In addition, non-vanishing of the Fourier coefficients is studied under the Ramanujan-Petersson conjecture.  
\end{abstract}

\section{Introduction}

Signs of Fourier coefficients of automorphic forms have been connected to classical problems in number theory and their study has attracted a lot of interest recently. For example, in \cite{Kowalski-Lau-Soundararajan-Wu} the authors related the question of finding the first sign change among the Fourier coefficients of a classical modular form to the question of finding the least quadratic non-residue modulo $d$, which has historically been a central problem in analytic number theory. More precisely, the connection is the following. Automorphic forms can be thought of as higher rank analogues of Dirichlet characters and it is natural to study whether properties of lower rank objects generalise to the higher rank setting. The problem of finding the least quadratic non-residue modulo $d$ concerns finding the least natural number $m$ so that $\chi_d(m)=-1$, where $\chi_d$ is the real Dirichlet character modulo $d$. It is not obvious how to formulate an analogue for this in the higher rank setting. Namely, the real Dirichlet character $\chi_d$ attains only two non-zero values ($+1$ and $-1$) and these (non-zero) values are distributed similarly as are independent random variables taking values $\pm 1$ with equal probability. Contrast this to Hecke eigenvalues (which match to certain Fourier coefficients when the underlying form is properly normalised) of Maass cusp forms for the group $\mathrm{SL}_n(\mathbb Z)$, which take more than two non-zero values and are (conjecturally in the higher rank setting) distributed according to the (generalised) Sato-Tate measure. Despite these differences, the authors were able to retain the analogy when considering signs of the Hecke eigenvalues instead of their values. In particular, the analogue for finding the least quadratic non-residue is to find the first negative Hecke eigenvalue. 

In addition to the work discussed above, the signs have been studied from many different points of view for both classical modular forms and their generalisations (see e.g. \cite{Kohnen-Lau-Shparlinski, Matomaki, Lau-Wu, Lau-Royer-Wu, Lester-Radziwill}). Besides finding the first negative Fourier coefficient, one can estimate the number of forms (in some family) for which the signs of the Fourier coefficients at primes lying in certain region are specified. Yet another interesting question is whether a primitive cusp form is determined uniquely by the sequence of signs of its Fourier coefficients at primes. Finally, one may ask how many sign changes does the sequence of Fourier coefficients have up to some point? For more complete survey on these topics and related results, see \cite{Lau-Liu-Wu2}. The last question is the one we study in the present article. We do not consider the situation where for instance $\lambda(m)>0$ and $\lambda(m+1)=0$ to be a sign change, because the sign of zero is undefined. Thus, by number of sign changes of a possibly vanishing sequence $\{\lambda(m)\}_m$ we mean the number of sign changes of $\lambda(m)$ on the subset of $m$ for which $\lambda(m)\neq 0$. 

For classical Hecke cusp forms (rank two case) Matom\"aki and Radziwi{\l\l} \cite{Matomaki-Radziwill1} have shown that there is a positive proportion of sign changes among the Fourier coefficients. Related to this, our aim is to establish a non-trivial lower bound for the number of sign changes for higher rank cusp forms. Specifically we will consider Hecke-Maass cusp forms for the group $\mathrm{SL}_3(\mathbb Z)$ (rank three case). Such forms admit a Fourier expansion \cite[Equation 6.2.1]{Goldfeld} with Fourier coefficients $A(m,n)$ indexed by pairs of integers. The most interesting among these coefficients are those of the form $A(m,1)$ as these are the Hecke eigenvalues when the underlying form is normalised so that $A(1,1)=1$. There is a subtle point that the coefficients $A(m,n)$ are not generally real-valued. Fortunately they are so in special instances, and thus the sign change problem is well-defined when this happens. What is relevant for the present paper is that the coefficients $A(m,1)$ are real-valued when the underlying form is self-dual and that the coefficients $A(m,m)$ are real-valued for any Hecke-Maass cusp form.

Signs of the Fourier coefficients of these higher rank Hecke-Maass cusp forms have been studied in recent years. In the general $\mathrm{GL}_n$-setting (where the Fourier coefficients are parametri-
sed by $(n-1)$-tuples of integers) finding the first negative Fourier coefficient for self-dual forms was considered by Liu, Qu, and Wu \cite{Liu-Qu-Wu}. For the group $\mathrm{SL}_3(\mathbb Z)$, Steiger \cite{Steiger} considered signs of the real part of $A(p,1)$ and proved (with his vertical Sato-Tate law) that there is a positive density of Hecke-Maass cusp forms $\phi$ such that the real part of\footnote{Sometimes we emphasise that $A(m,n)$ is a Fourier coefficient of the form $\phi$ by writing $A_\phi(m,n)$.} $A_\phi(p,1)$ is positive for all primes $p$ in some finite set. After that Xiao and Xu \cite{Xiao-Xu} considered the number of Hecke-Maass cusp forms $\phi$ such that $A_{\phi}(p,p)$ has a given sign for any prime $p$ lying in certain region. The latter result has since been generalised by Lau, Ng, Royer, and Wang \cite{Lau-Ng-Royer-Wang} for the group $\mathrm{SL}_n(\mathbb Z)$ with $n\geq 4$. In addition, the problem of finding Fourier coefficients of the same sign from short intervals was considered by Liu and Wu \cite{Liu-Wu} in the higher rank setting. 

In this article we are studying sign changes among the Fourier coefficients of Hecke-Maass cusp forms of the form $A(m,1)$ and $A(m,m)$. For self-dual forms it was first noticed by Qu \cite{Qu} that the sequence $\{A(m,1)\}_m$ has infinitely many sign changes. However, we are interested in quantitative results. Towards this the most optimal conjecture is that the sequence has a positive proportion of sign changes, i.e. that the number of sign changes of the sequence of Hecke eigenvalues $\{A(m,1)\}_{m\leq X}$ is of the order of magnitude $\#\{m\leq X\,:\,A(m,1)\neq 0\},$ where $\#\mathcal X$ denotes the cardinality of a set $\mathcal X$. Our goal in this paper is to give evidence towards such result by proving the following. Let $L(s,\phi)$ be the Godement-Jacquet $L$-function attached to the $\mathrm{GL}_3$ form $\phi$. Let $\theta\geq 0$ be such that $L(1/2+it,\phi)\ll_\varepsilon (3+|t|)^{\theta+\varepsilon}$ for every $\varepsilon>0$. Unconditionally we know that $\theta\leq 3/5$ when $\phi$ is self-dual \cite{ref2}. 

\begin{theorem}\label{merkkivaihtelu}
Assume that the underlying Hecke-Maass cusp form is self-dual.

(i) Suppose that $\theta>1/2$. Then the number of sign changes of the sequence of Hecke eigenvalues $\{A(m,1)\}_{m\leq X}$ is $\gg_\varepsilon X^{1/2\theta-\varepsilon}$ for any $\varepsilon>0$. In particular, for $\theta=3/5$ the sequence has $\gg_\varepsilon X^{5/6-\varepsilon}$ sign changes.

(ii) Suppose that $0\leq\theta\leq 1/2$. Then the sequence of Hecke eigenvalues $\{A(m,1)\}_{m\leq X}$ has $\gg_\varepsilon X^{1-\varepsilon}$ sign changes for any $\varepsilon>0$. 
\end{theorem}
\noindent The first part improves a result of Lau, Liu, and Wu \cite{Lau-Liu-Wu} who obtained $\gg X^{1/3}$ sign changes (which in turn improved an earlier result of Meher and Murty \cite{Meher-Murty}). We also note that more generally sign changes of the sequence $\{\Re A(m,1)\}_{m\leq X}$ have been considered by Hulse, Kuan, Lowry-Duda, and Walker \cite{Hulse-Kuan-Lowry-Duda-Walker} who obtained $\gg_\varepsilon X^{2/23-\varepsilon}$ sign changes in this case.




The proof of Theorem \ref{merkkivaihtelu} follows a different path compared to the method used in \cite{Lau-Liu-Wu}. Our strategy is to show that short interval $[x,x+H]$ contains a sign change for many $x\sim X$ (here and subsequently $x\sim X$ means $X\leq x\leq 2X)$ with $H$ as small as possible whereas Lau, Liu, and Wu proceeded by using a Voronoi type series approximation for the weighted mean of the Fourier coefficients together with the method of Heath-Brown and Tsang \cite{Heath-Brown-Tsang}. For finding sign changes in short intervals, we follow an approach of Matom\"aki and Radziwi\l\l\,\cite{Matomaki-Radziwill1}. That is, we will compare the quantities 
\begin{align*}
S_1(x;H):=\left|\sum_{\substack{x\leq mk\leq x+H\\
m\sim M\\
(k,m)=1}}A(mk,1)\right|\quad\text{and}\quad S_2(x;H):=\sum_{\substack{x\leq mk\leq x+H\\
m\sim M\\
(k,m)=1}}|A(mk,1)|,
\end{align*}
where $M<H$ is a small power of $X$. Obviously the value of the former sum is at most the value of the latter sum and they are equal precisely when there is no sign change in the interval $[x,x+H]$. Our aim is now to show that for some small $H$ we have $S_1(x;H)<S_2(x;H)$ for $\gg_\varepsilon X^{1-\varepsilon}$ of $x\sim X$ and so there is a sign change in the interval $[x,x+H]$ for such $x\sim X$. This clearly implies that there are at least $\gg_\varepsilon X^{1-\varepsilon}H^{-1}$ sign changes for any $\varepsilon>0$. Unconditionally the smallest value of $H$ we can take turns out to be $\approx X^{1/6}$ (more generally, for $\theta>1/2$ the optimal choice is $H\approx X^{1-1/2\theta}$ and for $0\leq\theta\leq 1/2$ we can choose $H$ to be an arbitrarily small power of $X$), which gives the desired result. We emphasise that the bilinear structure in the sums $S_1(x;H)$ and $S_2(x;H)$ leads to an improved exponent. Indeed, unconditionally our argument would only yield $\gg_\varepsilon X^{2/3-\varepsilon}$ sign changes for the Hecke eigenvalues when comparing the quantities
\begin{align*}
\left|\sum_{x\leq m\leq x+H}A(m,1)\right|\qquad\text{and}\qquad \sum_{x\leq m\leq x+H}\left|A(m,1)\right|.
\end{align*}
This matches the amount of sign changes which follows from the convexity bound of $L(s,\phi)$. 

There are some differences compared to the treatment in \cite{Matomaki-Radziwill1} in which the authors reduce bounding an analogue of $S_1(x;H)$ for almost all $x\sim X$ to a shifted convolution problem involving the Fourier coefficients of the underlying form. This also allows them to work with very small $H$ and in addition to mollify, which saves some powers of logarithm, thus giving the optimal result. This strategy is problematic in the higher rank setting due to lack of progress towards the corresponding shifted convolution problem. Instead, we establish a non-trivial upper bound for $S_1(x;H)$ for almost all $x\sim X$ by using complex analytic techniques in the spirit of \cite{Matomaki-Radziwill2}. This is the main reason why we need to take $H$ to be fairly large.     

We will also show that one gets a positive proportion of sign changes for many $\mathrm{GL}_3$ forms. Let $\mathcal H_T$ denote the set of self-dual Hecke-Maass cusp forms $\phi$ for the group $\mathrm{SL}_3(\mathbb Z)$ with\footnote{Here $\nu_\phi=(\nu_1,\nu_2)\in\mathbb C^2$ is the spectral parameter of $\phi$ and $\|\nu_\phi\|:=\sqrt{|2\nu_1+\nu_2|^2+|\nu_2-\nu_1|^2+|\nu_1+2\nu_2|^2}$.} $\|\nu_\phi\|\leq T$. For self-dual forms we show the following result stating that the sequence $\{A_\phi(m,1)\}_m$ has a positive proportion of sign changes for at least $(1-\varepsilon)\#\mathcal H_T$ forms in $\mathcal H_T$ for any fixed $\varepsilon>0$ as $T\longrightarrow\infty$. This follows from combining a result of Matom\"aki and Radziwi{\l\l}\, \cite[Corollary 3]{Matomaki-Radziwill2} together with an effective Sato-Tate result of Lau and Wang \cite{Lau-Wang} (which is a real-analytic analogue for the result of Murty and Sinha \cite{Murty-Sinha}) after reducing to the $\mathrm{GL}_2$-case.

\begin{theorem}\label{Hecke-aa}
Let $\varepsilon>0$ be fixed. Then there exists a subset $\mathcal S_T\subset\mathcal H_T$ with at least $(1-\varepsilon)\#\mathcal H_T$ elements such that for any $\phi\in\mathcal S_T$ the sequence $\{A_\phi(m,1)\}_m$ has a positive proportion of sign changes as $T\longrightarrow\infty$. 
\end{theorem}

\noindent On the other hand, we will also consider sign changes among the Fourier coefficients of $\mathrm{GL}_3$ forms, which are not necessarily lifts of lower rank forms (like the self-dual forms are). In this respect a good analogue for the $\mathrm{GL}_2$ Fourier coefficients are the multiplicative coefficients $A(m,m)$, which are also known to be always real-valued. It follows from the Hecke relations that $A(p,p)=|A(p,1)|^2-1$ at primes $p$, so these coefficients are also closely related to the Hecke eigenvalues.

For these coefficients we have an analogous result to Theorem \ref{Hecke-aa}, but the proof requires more work as we cannot reduce to the lower rank case due to the fact that there is no relation between generic $\mathrm{GL}_3$ cusp forms and $\mathrm{GL}_2$ cusp forms. Instead, the proof is based on a weaker variant of the Lau-Wang result \cite{Lau-Wang} in the $\mathrm{GL}_3$-setting we establish, see Theorem \ref{Sato-Tate} below. To state the result concerning the positive proportion of sign changes we need some notation. We write $\mathfrak a_\mathbb C^*$ for the dual of the Lie algebra $\mathfrak a_\mathbb C$ of the subgroup of diagonal matrices inside $\mathrm{SL}_3(\mathbb C)$ and $W$ for the Weyl group of $\mathrm{SL}_3(\mathbb C)$ acting on $\mathbb C^3$ by permuting the elements. Let $\nu_0\in i\mathfrak a_\mathbb C^*$ be fixed, $\eta>0$ be very small but fixed, and set\footnote{Here we identify $\nu=(\nu_1,\nu_2)\in\mathbb C^2$ with $(2\nu_1+\nu_2,\nu_2-\nu_1,-\nu_1-2\nu_2)\in\mathbb C^3$.}
\begin{align*}
\widetilde{\mathcal H}_T:=\left\{\phi\,\,\text{Hecke-Maass cusp form for }\mathrm{SL}_3(\mathbb Z)\,\text{so that }\|w.\nu_\phi-T\nu_0\|\leq T^{1-\eta}\,\text{for some }w\in W\right\}.
\end{align*}
Then we have the following.
\begin{theorem}\label{aa}  
Let $\varepsilon>0$ be fixed. Then there exists a subset $\widetilde{\mathcal S}_T\subset\widetilde{\mathcal H}_T$ with at least $(1-\varepsilon)\#\widetilde{\mathcal H}_T$ elements such that for any $\phi\in\widetilde{\mathcal S}_T$ the sequence $\{A_\phi(m,m)\}_m$ has a positive proportion of sign changes as $T\longrightarrow\infty$. 
\end{theorem}

\noindent One can also pose a question about the number of non-zero Fourier coefficients, that is to ask how large the set $\{m\leq X:\, A(m,1)\neq 0\}$ is? Unconditionally we know that this quantity is $\gg_\varepsilon X^{1-\varepsilon}$ due to a result of Jiang and L\"u \cite{JiangLu}. Slightly better lower bounds are known for $\mathrm{GL}_2$ Hecke-Maass cusp forms by comparing the second and fourth moments of the Fourier coefficients (and for holomorphic cusp forms the optimal result is known, see \cite{Serre}). One can show that the following strengthening holds for almost all forms under the Ramanujan-Petersson conjecture. By this we mean that the set of exceptional forms has a zero density. 

\begin{theorem}\label{non-vanishing-theorem}
(i) Assume the Ramanujan-Petersson conjecture for classical Hecke-Maass cusp forms. Then 
\begin{align*}
\#\{m\leq X\,:\,A_\phi(m,1)\neq 0\}\asymp X\prod_{\substack{p\leq X\\
A_\phi(p,1)=0}}\left(1-\frac 1p\right)
\end{align*}
for almost all self-dual Hecke-Maass cusp forms $\phi\in\mathcal H_T$ as $T\longrightarrow\infty$.

(ii) Assume the Ramanujan-Petersson conjecture for $\mathrm{GL}_3$ Hecke-Maass cusp forms. Then 
\begin{align*}
\#\{m\leq X\,:\,A_\phi(m,m)\neq 0\}\asymp X\prod_{\substack{p\leq X\\
A_\phi(p,p)=0}}\left(1-\frac 1p\right)
\end{align*}
for almost all Hecke-Maass cusp forms $\phi\in\widetilde{\mathcal H}_T$ as $T\longrightarrow\infty$.
\end{theorem}

\begin{remark}
Note that the first part is conditional on the Ramanujan-Petersson conjecture for classical Hecke-Maass cusp forms whereas the latter part is conditional on the higher rank conjecture. The reason for this is the relation connecting $A(p,1)$ to the Fourier coefficient of a classical Hecke-Maass cusp form at prime $p$ in the self-dual case and the relation connecting $A(p,p)$ to $A(p,1)$ in the general case.
\end{remark}

\noindent It follows from sieve theoretic arguments (see \cite[Lemma 3.1.]{Matomaki-Radziwill1}) that the statement of the above theorem holds for the form $\phi$ if $A_\phi(p,1)<0$ for a positive proportion of primes $p$. This is what we shall show to hold for almost all forms. Such result follow from rather simple applications of Elliott-Montgomery-Vaughan type large sieve inequalities described below together with the Hecke relations. Considerations needed for this together with the work of Matom\"aki and Radziwi{\l\l} \cite[Corollary 3]{Matomaki-Radziwill2} lead to the following result as a by-product.

\begin{theorem}\label{positive-negative}
Let $\varepsilon>0$ be fixed. Assume the Ramanujan-Petersson conjecture for classical Hecke-Maass cusp forms. Then there exists a subset $\mathcal S'_T\subset\mathcal H_T$ with at least $(1-\varepsilon)\#\mathcal H_T$ elements such that for any $\phi\in\mathcal S'_T$ asymptotically half of the non-zero coefficients $A_\phi(m,1)$ are positive and half of them are negative as $T\longrightarrow\infty$. 

Similarly, there exists a subset $\mathcal{S}''_T\subset\widetilde{\mathcal H}_T$ with at least $(1-\varepsilon)\#\widetilde{\mathcal H}_T$ elements such that for any $\phi\in\mathcal{S}''_T$ asymptotically half of the non-zero coefficients $A_\phi(m,m)$ are positive and half of them are negative as $T\longrightarrow\infty$ assuming the Ramanujan-Petersson conjecture for $\mathrm{GL}_3$ Hecke-Maass cusp forms. 
\end{theorem}  

\noindent It is natural to wonder whether the results of the present article generalise to the higher rank setting, namely for the group $\mathrm{SL}_n(\mathbb Z)$ with $n\geq 4$. Unconditionally the method used to obtain Theorem \ref{merkkivaihtelu} does not yield any non-trivial lower bound for the number of sign changes for higher rank groups. There are two reasons for this. Firstly, the mean-value theorem for Dirichlet polynomials (Lemma \ref{MVT}) is too weak in this situation and the other reason is that in general we do not know as good lower bound for the sum of absolute values of Hecke eigenvalues as given in Lemma \ref{Jiang-Lu} in the $\mathrm{GL}_3$-setting (see Remark \ref{lower-bound}). However, under the generalised Lindel\"of hypothesis and the Ramanujan-Petersson conjecture one obtains the lower bound $\gg x/(\log x)^{1-1/n}$ for such sum when $n\geq 4$. Actually, more precise estimate follows assuming the generalised Sato-Tate conjecture;
\begin{align*}
\sum_{m\leq x}|A(m,1,...,1)|\gg\frac{x}{(\log x)^{k_n}}
\end{align*}
for some $0<k_n<1-1/n$ depending on $n$, see the introduction in \cite{Jiang-Lu-Wang}. Such results seem currently out of reach unconditionally even assuming that the underlying form is self-dual. 

Theorem \ref{Hecke-aa} in turn relies crucially on a bijection between self-dual Maass cusp forms on $\mathrm{GL}_3$ and non-dihedral Maass cusp forms on $\mathrm{GL}_2$. Similar relationship does not exist in the higher rank setting and thus an alternative method is needed in order to establish a possible generalisation.

Concerning the coefficients $A(m,m)$ one would need an effective Sato-Tate theorem generalising Theorem \ref{Sato-Tate}, which in turn is based on Lemma \ref{asymptotic} below proved by Buttcane and Zhou \cite{Buttcane-Zhou}. This lemma is thought to generalise, but it seems currently out of reach (however, for a closely related result, see \cite[Theorem 1.5.]{Matz-Templier}). The underlying feature in this result is the orthogonality of Fourier coefficients discussed for instance in \cite{Zhou}. The orthogonality relation in question is actually known for the group $\mathrm{SL}_4(\mathbb Z)$ \cite{Goldfeld-Stade-Woodbury} and therefore there might be some hope to generalise the result of Buttcane and Zhou to this setting, which in turn would imply a generalisation of Theorem \ref{aa} for the group $\mathrm{SL}_4(\mathbb Z)$ (possibly under the Ramanujan-Petersson conjecture). 

The last two theorems on non-vanishing are based on the Hecke relations and large sieve inequalities of Elliott-Montgomery-Vaughan type. As the Hecke relations are available in more general settings as is the large sieve \cite{Lau-Ng-Royer-Wang} (a variant with restriction to self-dual forms as in Lemma \ref{self-dual-large-sieve} should be possible to establish too), these results should generalise in a straightforward manner, except combinatorics of the Hecke algebra becomes more complicated. However, we do not pursue this extension here. 

\section{Notation}

The symbols $\ll$, $\gg$, $\asymp$, $O$, and $o$ are used for the usual asymptotic notation: for
complex valued functions $f$ and $g$ in some set $X$, the notation $f \ll g$ means that
$|f(x)|\leq C|g(x)|$ for all $x\in X$ for some implicit constant $C\in\mathbb R_+$. When the
implied constant depends on some parameters $\alpha,\beta,...$ we use $\ll_{\alpha,\beta},...$ instead
of mere $\ll$. The notation $g\gg f$ means $f\ll g$, and $f\asymp g$ means $f\ll g\ll f$. Moreover, $f=O(g)$ means the same as $f\ll g$. When $f$ and $g$ depend on $M$, we say that $f(M)=o(g(M))$ if $g(M)$ never vanishes and $f(M)/g(M)\longrightarrow 0$ as $M\longrightarrow\infty$.

We write $\varepsilon$ for an arbitrary small positive constant, which may differ from line to line unless specified otherwise. As usual, complex variables are written in the form $s=\sigma+it$ with $\sigma,t\in\mathbb R$. The subscript in the integral $\int_{(\sigma)}$ means that we integrate over the vertical line $\Re s=\sigma$. 
We write $x\sim X$ for $x\in[X,2X]$. The characteristic function of a set $\mathcal X$ is denoted by $1_{\mathcal X}$, the cardinality is $\#\mathcal X$, and the complement is denoted by $\mathcal X^c$. Finally, let $\mu$ be the M\"obius function and $d_r$ be the $r$-fold divisor function. 
 
\section{Useful results}
We start by recalling a few facts about higher rank automorphic forms, Hecke operators, and automorphic $L$-functions following \cite{Blomer,Goldfeld}. Maass cusp forms are eigenfunctions of the commutative algebra $\mathcal D$ of invariant differential operators of $\mathrm{SL}_3(\mathbb R)$ acting on the space $L^2(\mathrm{SL}_3(\mathbb R)/\mathrm{SO}_3)$, which is generated by two elements (see \cite[Chapters 2.3 and 2.4]{Goldfeld}), the Laplacian and another operator of degree $3.$ One class of eigenfunctions of $\mathcal D$ is given by the power functions $I_{\nu_1,\nu_2}$, which is parametrised by two complex numbers $\nu_1$ and $\nu_2$ (see \cite[Chapter 5]{Goldfeld}). A Maass cusp form $\phi$ for the group $\mathrm{SL}_3(\mathbb Z)$ with spectral parameters $\nu_1,\nu_2$ is an element in $L^2(\mathrm{SL}_3(\mathbb Z)\backslash\mathrm{SL}_3(\mathbb R)/\mathrm{SO}_3)$ that is an
eigenfunction of $\mathcal D$ with the same eigenvalues as $I_{\nu_1,\nu_2}$, and vanishes along all parabolics, that is,
\begin{align*}
\int_{(\mathrm{SL}_3(\mathbb Z)\cap U)\backslash U}\phi(uz)\mathrm d u = 0
\end{align*}
for $U=\left\{\begin{pmatrix}
1 & & *\\
 & 1 & *\\
 & & 1
 \end{pmatrix}\right\}, \left\{\begin{pmatrix}
1 & * & *\\
 & 1 & \\
 & & 1
 \end{pmatrix}\right\} $,
and 
$\left\{\begin{pmatrix}
1 & * & *\\
 & 1 & *\\
 & & 1
 \end{pmatrix}\right\} $. 

This can be re-phrased in more representation theoretic terms. Let $\mathrm{SL}_3(\mathbb R)=NAK$ be the Iwasawa
decomposition, where $K:=\mathrm{SO}_3$, $N$ is the standard unipotent subgroup, and $A$ is the group of diagonal
matrices with determinant one and positive entries. An infinite-dimensional, irreducible, everywhere unramified cuspidal automorphic representation $\pi$ of $\mathrm{GL}_3(\mathbb A_{\mathbb Q})$ with a trivial central character is generated by a Hecke-Maass cusp form $\phi$ for the group $\mathrm{SL}_3(\mathbb Z)$ as above. This representation factorises as an infinite restricted tensor product of local representations $\pi=\otimes_{v\leq\infty}'\pi_v$. The representation $\pi_\infty$ at the archimedean place is induced from the parabolic subgroup $NA$ by the extension of a character $\chi:\, A \longrightarrow\mathbb C^\times$, $\text{diag}(x_1,x_2,x_3)\mapsto x_1^{\alpha_1}x_2^{\alpha_2}x_3^{\alpha_3}$ with $\alpha_1+\alpha_2+\alpha_3=0$. In this way we can identify the spherical cuspidal automorphic spectrum with a discrete subset of the Lie algebra $\mathfrak a_{\mathbb C}^*/ W$, where we associate to each Maass cusp form $\phi$ the linear form $\ell= (\alpha_1,\alpha_2,\alpha_3)\in\mathfrak a_{\mathbb C}^*/ W$ called the Langlands parameter.  A convenient basis for $\mathfrak a^*$ is given by the fundamental weights $\text{diag}(2/3,-1/3,-1/3)$, $\text{diag}(1/3,1/3,-2/3)$ of $\mathrm{SL}_3$. The coefficients of $\ell=(\alpha_1,\alpha_2,\alpha_3)$ with respect to this basis can be obtained by evaluating $\ell$ at the two co-roots $\text{diag}(1,-1,0)$, $\text{diag}(0,1,-1)\in\mathfrak a$ and they are given by $3\nu_1,3\nu_2$. In this case we say that $(\nu_1,\nu_2)\in\mathbb C^2$ is the type of the form or that $\nu_1$ and $\nu_2$ are the spectral parameters of the form. We then have the following relations between the Langlands parameters and the spectral parameters: $\alpha_1= 2\nu_1+\nu_2$, $\alpha_2= -\nu_1 + \nu_2$, $\alpha_3= -\nu_1-2\nu_2$. It is also convenient to define $\nu_3:=-\nu_1-\nu_2$. For any finite prime $p$, the local representation $\pi_p$ has a $\mathrm{GL}_3(\mathbb Z_p)$-fixed vector and the space of such vectors is one-dimensional. The local Hecke algebra $\mathcal H_p$ acts on this space of invariant vectors. The Satake isomorphism tells that the algebra $\mathcal H_p$ has three generators and the eigenvalues of these generators when acting on the space of $\mathrm{GL}_3(\mathbb Z_p)$-fixed vectors are called the Satake parameters of $\phi$ at $p$. They will be denoted by $\alpha_{1,p}(\phi)$, $\alpha_{2,p}(\phi)$ and $\alpha_{3,p}(\phi)$. Sometimes we will write $\alpha_\phi(p):=(\alpha_{1,p}(\phi), \alpha_{2,p}(\phi),\alpha_{3,p}(\phi))\in\mathbb C^3$.

With the above normalisation, $\phi$ is an eigenform of the Laplacian with eigenvalue
\begin{align*}
1-3\nu_1^2-3\nu_1\nu_2-3\nu_2^2 = 1-\frac12(\alpha_1^2+\alpha_2^2+\alpha_3^2).
\end{align*}
The Ramanujan-Petersson conjecture predicts that the Langlands parameters $\alpha_1,\alpha_2,\alpha_3$ of Maass cusp forms are purely imaginary (equivalently, the spectral parameters $\nu_1,\nu_2$ are purely imaginary). A Maass cusp form is called exceptional if it violates the Ramanujan-Petersson conjecture. An equivalent formulation of the Ramanujan-Petersson conjecture in terms of the Fourier coefficients is that $|A(p,1)|\leq 3$ for any prime $p$ (similarly, in the general $\mathrm{GL}_n$-setting the conjecture states $|A(p,1,...,1)|\leq n$ for all primes $p$). We shall write bounds towards the Ramanujan-Petersson conjecture as $|A(p,1)|\leq 3p^\vartheta$ for some $\vartheta\geq 0$. Currently it is only known that $\vartheta\leq 5/14$ unconditionally \cite{Kim-Sarnak}, but the Ramanujan-Petersson conjecture of course predicts that $\vartheta=0$ is admissible (for classical Maass cusp forms we currently know that $|\lambda(p)|\leq 2p^{7/64}$ unconditionally). However, for any fixed prime $p$, the Ramanujan-Petersson conjecture in $\mathrm{GL}_3$ is known to hold for almost all forms meaning that the density of the set of exceptional forms is zero as was recently shown by Lau, Ng, and Wang \cite{Lau-Ng-Wang-2020}. We formulate a version of their result, which is suitable for our needs. To state this we need some notation. Let $\Omega\subset i\mathfrak a_{\mathbb C}^*$ be a compact Weyl-group invariant subset disjoint from the Weyl chamber walls $\{\mu\in\mathfrak a_{\mathbb C}^*:w.\mu=\mu\,\text{for some }w\in W,w\neq 1\}$ and $T>1$ be a large parameter. Fix some $\nu_0\in\Omega$. For $\nu\in\mathfrak a_{\mathbb C}^*$ we set $\psi(\nu):=\exp(3(\nu_1^2+\nu_2^2+\nu_3^2))$ and
\begin{align*}
P(\nu):=\prod_{0\leq n\leq A}\prod_{j=1}^3\frac{\nu_j^2-\frac19(1+2n)^2}{T^2}
\end{align*}
for some large, fixed $A$ to compensate poles of the spectral measure in a large tube. Now we define the weight function
\begin{align}\label{weight-function}
h_T(\nu):=P(\nu)^2\left(\sum_{w\in W}\psi\left(\frac{w.\nu-T\nu_0}{T^{1-\eta}}\right)\right)^2
\end{align}
for some very small $\eta>0$. Note that $h_T$ is non-negative and $h_T(\nu_\phi)\asymp 1$ for $\phi\in\widetilde{\mathcal H}_T$. We write $X_p$ for the set of Hecke-Maass cusp forms for the group $\mathrm{SL}_3(\mathbb Z)$ not satisfying the Ramanujan-Petersson conjecture at $p$. Then it follows from \cite{Lau-Ng-Wang-2020} that
\begin{align}\label{exceptions}
\sum_{\phi\in X_p}h_T(\nu_\phi)\ll\left(\sum_\phi h_T(\nu_\phi)\right)\left(\frac{\log p}{\log T}\right)^3.
\end{align} 
Let $\phi$ be a Maass cusp form of type $(\nu_1,\nu_2)\in\mathbb C^{2}$ for the group $\mathrm{SL}_3(\mathbb Z)$. By analogue to the classical situation, it follows that for every integer $m\geq 1$, there is a Hecke operator given by
\begin{align*}
T_m\phi(z):=\frac1{m^{5/2}}\sum_{\substack{\prod_{\ell=1}^3 c_\ell=m\\
0\leq c_{i,\ell}<c_\ell\,\,(1\leq i<\ell\leq 3)}}\phi\left(\begin{pmatrix}
c_1 & c_{1,2} & c_{1,3}\\
 & c_2 & c_{2,3}\\
 & & c_3
\end{pmatrix}\cdot z\right)
\end{align*}
acting on the space $L^2(\mathrm{SL}_3(\mathbb Z)\backslash\mathrm{SL}_3(\mathbb R)/\mathrm{SO}_3)$ of square-integrable automorphic functions. Unlike in the classical situation, these operators are not self-adjoint, but they are normal. If a Maass cusp form $\phi$ is an eigenfunction of every Hecke operator, it is called a Hecke-Maass cusp form. We remark that if the Fourier coefficient $A(1,1)$ is zero, then the form vanishes identically. For more about the theory of Hecke operators for the group $\mathrm{SL}_3(\mathbb Z)$, see \cite[Section 9.3.]{Goldfeld}.

Fourier coefficients and Satake parameters of a given Hecke-Maass cusp form are closely related by the work of Shintani \cite{Shintani} together with results of Casselman and Shalika \cite{Casselman-Shalika}. They showed that for any prime number $p$ and $\beta_1,\beta_2\in\mathbb Z_+\cup\{0\}$ one has
\begin{align}\label{Shintani}
A_\phi(p^{\beta_1},p^{\beta_{2}})=S_{\beta_{2},\beta_1}(\alpha_{1,p}(\phi),\alpha_{2,p}(\phi),\alpha_{3,p}(\phi)), 
\end{align} 
where 
\begin{align}\label{Schur}
S_{\beta_{2},\beta_1}(x_1,x_2,x_3):=\frac1{V(x_1,x_2,x_3)}\det\left[\begin{pmatrix}
x_1^{2+\beta_{2}+\beta_{1}} & x_2^{2+\beta_{2}+\beta_{1}} & x_3^{2+\beta_{2}+\beta_{1}}\\
x_1^{1+\beta_{2}} & x_2^{1+\beta_{2}} & x_3^{1+\beta_{2}}\\
1 & 1 & 1
\end{pmatrix}
\right]
\end{align}
is a Schur polynomial, and $V(x_1,x_2,x_3)$ is the Vandermonde determinant given by
\begin{align*}
V(x_1,x_2,x_3):=\prod_{1\leq i<j\leq 3}(x_i-x_j).
\end{align*}
Next, we define an important notion of a dual Maass cusp form. Let 
\begin{align*}
\widetilde \phi(z):=\phi(w\cdot\,^t(z^{-1})\cdot w),\quad\text{where} \qquad w:=\begin{pmatrix}
& & -1  \\
& 1 & \\
 1 & & 
\end{pmatrix}.   
\end{align*}
Then $\widetilde \phi$ is a Maass cusp form of type $(\nu_{2},\nu_1)\in\mathbb C^{2}$ for the group $\mathrm{SL}_3(\mathbb Z)$ and it is called the dual Maass cusp form of $\phi$. We say that Hecke-Maass form $\phi$ is self-dual if $\phi=\widetilde\phi$. It turns out that 
\begin{align}\label{dualmaasscoeff}
A_\phi(m_1,m_{2})=A_{\widetilde \phi}(m_{2},m_1)
\end{align}
for every $m_1,m_{2}\geq 1$. 

The Fourier coefficients of a Hecke-Maass cusp form satisfy the multiplicativity relation
\begin{align}\label{Hecke-relation1}
A(m,1)A(m_1,m_{2})=\sum_{\substack{\prod_{\ell=1}^3 c_\ell=m\\
c_j|m_j\,\,\text{for }1\leq j\leq 2}}A\left(\frac{m_1c_3}{c_1},\frac{m_2c_1}{c_2}\right),
\end{align}
which, together with M\"obius inversion, gives
\begin{align}\label{Mobius-relation}
A(m_1,m_2)=\sum_{d|(m_1,m_2)}\mu(d)A\left(\frac{m_1}d,1\right)A\left(1,\frac{m_2}d\right)
\end{align}
for positive integers $m,m_1$, and a non-negative integer $m_{2}$. Furthermore, the relation
\begin{align*}
A(m_1,m_{2})A(m_1',m_{2}')=A(m_1m_1',m_{2}m_{2}')
\end{align*}
holds if $(m_1 m_{2},m_1'm_{2}')=1$. In particular, the coefficients $A(m,1)$ and $A(m,m)$ are multiplicative. For the proofs of these facts, see \cite[Theorem 9.3.11.]{Goldfeld}

For a Hecke eigenfunction, one can use M\"obius inversion to show that the relation
\begin{align*}
A(m_1,m_{2})=\overline{A(m_2,m_1)}
\end{align*} 
holds \cite[Theorem 9.3.6, Theorem 9.3.11, Addendum]{Goldfeld}. In particular, together with the relation (\ref{dualmaasscoeff}) this yields that 
\begin{align*}
\overline{A_\phi(m,1)}=A_{\widetilde \phi}(m,1),
\end{align*}
and consequently the Hecke eigenvalues $A(m,1)$ are real-valued for self-dual forms. The fact that $A(m,m)$ is always real-valued follows from (\ref{Mobius-relation}).

Associated to the form $\phi$ is the $L$-series given by
\begin{align*}
L(s,\phi):=\sum_{m=1}^\infty\frac{A_\phi(m,1)}{m^s},
\end{align*}
which converges for $\sigma>1$. This has an entire continuation to the whole complex
plane via the functional equation
\begin{align}\label{functional-equation}
L(s,\phi)=\pi^{3s-3/2}\frac{G(1-s,\widetilde \phi)}{G(s,\phi)} L(1-s,\widetilde \phi),
\end{align}
where
\begin{align*}
G(s,\phi):=\prod_{j=1}^3\Gamma\left(\frac{s-\alpha_j}2\right)\quad\text{and so }\quad G(s,\widetilde \phi)=\prod_{j=1}^3\Gamma\left(\frac{s-\widetilde\alpha_j}2\right).
\end{align*}
Recall that here $\alpha_j$ and $\widetilde\alpha_j$ are the Langlands parameters of $\phi$ and $\widetilde \phi$, respectively. This
defines an $L$-function attached to the form $\phi$ called the Godement-Jacquet $L$-function. For more information, see \cite[Chapter 6.5.]{Goldfeld}. The generalised Lindel\"of hypothesis predicts that 
\begin{align*}
L\left(\frac12+it,\phi\right)\ll_\varepsilon(3+|t|)^{\varepsilon}
\end{align*} 
for any $\varepsilon>0$. Currently we know that 
\begin{align*}
L\left(\frac12+it,\phi\right)\ll_\varepsilon(3+|t|)^{3/5+\varepsilon}
\end{align*}
for self-dual $\phi$ \cite{ref2}.

In what follows we will make use of the following result of Jiang and L\"u \cite{JiangLu}.

\begin{lemma}\label{Jiang-Lu}
For a self-dual Hecke eigenform we have
\begin{align*}
\sum_{m\leq x}|A(m,1)|\gg_\varepsilon x^{1-\varepsilon}
\end{align*} 
for any $\varepsilon>0$.
\end{lemma}

\begin{remark}\label{lower-bound}
This improves the bound $\gg_\varepsilon x^{1-\vartheta-\varepsilon}$, which follows from the Rankin-Selberg theory and the pointwise bound towards the Ramanujan-Petersson conjecture. 
\end{remark}
\noindent The proof of Theorem \ref{Hecke-aa} is based on a reduction to the lower rank setting. This is achieved by the following result of Ramakrishnan \cite{Ramakrishnan}. 

\begin{lemma} There is a bijection between the set of non-dihedral Maass cusp forms for the group $\mathrm{SL}_2(\mathbb Z)$ and the set of self-dual Maass cusp forms for the group $\mathrm{SL}_3(\mathbb Z$) given by the symmetric square lift. 
\end{lemma}

\begin{remark}
This follows from Ramakrishnan's main result as cusp forms we consider are assumed to have a trivial central character in which case the adjoint square lift is identical to the symmetric square lift.
\end{remark}

\noindent Moreover, if a self-dual Hecke-Maass cusp form $\phi$ on $\mathrm{GL}_3$ is a symmetric square lift of a Hecke-Maass cusp form $g$ on $\mathrm{GL}_2$, then their Fourier coefficients at primes are related by the relation $A_\phi(p,1)=\lambda_g(p^2)=\lambda_g(p)^2-1$.   

Let us then discuss the measures involved in this article following \cite{Buttcane-Zhou}. Let $\Phi$ be the root system of $\mathrm{SL}_3(\mathbb C)$ and $\Phi^+$ be the set of positive roots. Write $\rho:=\frac12\sum_{\alpha\in\Phi^+}\alpha$ for the half-sum of positive roots. Let $\mathcal P^+$ be the positive Weyl chamber of dominant weights. Finally, let $p$ be a prime number and let $\mathrm d s$ be the normalised Haar measure on
\begin{align*}
T_0:=\left\{(e^{i\theta_1},e^{i\theta_2},e^{i\theta_3}):\,e^{i(\theta_1+\theta_2+\theta_3)}=1\right\}\subset\mathrm{SL}_3(\mathbb C).
\end{align*}

\noindent Then the (generalised) Sato-Tate measure on $T_0/W$ is given by
\begin{align*}
\mathrm d \mu_\infty:=\frac1{|W|}\prod_{\beta\in\Phi}\left(1-e^\beta(s)\right)\mathrm d s.
\end{align*}
The $p$-adic Plancherel measure on $\mathrm{SL}_3(\mathbb C)$, which is supported on $T_0/W$, is given by
\begin{align}\label{p-adic-plancherel}
\mathrm d \mu_p:=\frac{W(p^{-1})}{\prod_{\beta\in\Phi}\left(1-p^{-1}e^{\beta}(s)\right)}\mathrm d \mu_\infty
\end{align}
with
\begin{align*}
W(q):=\sum_{w\in W}q^{\text{length}(w)},
\end{align*} 
where the length of $w\in W$ is the smallest $\ell\in\mathbb N$ such that $w$ is a product of $\ell$ reflections by simple roots. 

Explicitly, integration with respect to these measures is given by
\begin{align*}
\int_{T_0/W}f\,\mathrm d \mu_\infty=\frac1{24\pi^2}\int\limits_0^{2\pi}\int\limits_0^{2\pi}f(\theta_1,\theta_2)\prod_{1\leq \ell<j\leq 3}\left|e^{i\theta_\ell}-e^{i\theta_j}\right|^2\,\mathrm d\theta_1\mathrm d\theta_2
\end{align*}
and
\begin{align}\label{explicit-formula}
\int_{T_0/W}f\,\mathrm d \mu_p=\frac{(1-p^{-2})(1-p^{-3})}{6(1-p^{-1})^2}\int\limits_0^{2\pi}\int\limits_0^{2\pi}f(\theta_1,\theta_2)\prod_{1\leq \ell<j\leq 3}\left|\frac{e^{i\theta_\ell}-p^{-1}e^{i\theta_j}}{e^{i\theta_\ell}-e^{i\theta_j}}\right|^{-2}\mathrm d \theta_1\mathrm d \theta_2
\end{align}
for any continuous function $f:T_0/W\longrightarrow\mathbb C$, respectively. Here we have identified $T_0/W$ with a subset of $\{(e^{i\theta_1},e^{i\theta_2}):\theta_1,\theta_2\in[0,2\pi]\}$. For discussion on the support of the measure $\mathrm d \mu_p$, see \cite{Blomer-Buttcane-Raulf}.

Next we need some facts about Kazhdan-Lusztig polynomials. For a symbol $q$ and weight $\beta$ we define the Kostant $q$-partition by
\begin{align*}
P_q(\beta):=\sum_{\substack{\beta=\sum n(\alpha)\alpha\\
\alpha\in\Phi^+,\,n(\alpha)\geq 0}}q^{\sum n(\alpha)}
\end{align*}
and the Kazhdan-Lusztig $q$-polynomial by
\begin{align*}
\mathfrak M_\lambda^\beta(q):=\sum_{w\in W}(-1)^{\text{length}(w)}P_q(w.(\lambda+\rho)-(\beta+\rho))
\end{align*}
for any root $\lambda$.

Let $\ell_1,\ell_2\geq 0$ be natural numbers. We write $\aleph(\ell_2,\ell_1):=\lambda_1\ell_1+\lambda_2\ell_2$ with $\lambda_1$ being the highest weight in $\mathcal P^+$ for the standard inclusion $\mathrm{SL}_3(\mathbb C)\hookrightarrow\mathrm{GL}_3(\mathbb C)$ and $\lambda_2$ is the highest weight in $\mathcal P^+$ for the exterior power representation $\wedge^2 V_{\lambda_1}$. A crucial fact for us is that $\mathfrak M_{\aleph(\ell_2,\ell_1)}^0(p^{-1})$ is related to Schur polynomials by the formula 
\begin{align}\label{Kato}
\mathfrak M_{\aleph(\ell_2,\ell_1)}^0(p^{-1})=\int_{T_0/W}S_{\ell_1,\ell_2}(\sigma)\,\mathrm d \mu_p(\sigma),
\end{align}
which can be found in the work of Kato \cite{Kato}.

We recall the following result due to Buttcane and Zhou \cite{Buttcane-Zhou}. The spectral measure, which counts the number of Maass cusp forms, is given by
\begin{align*}
\text{spec}(\nu)\,\mathrm d \nu:=\frac3{256\pi^5}\prod_{j=1}^3\left(3\nu_j\tan\left(\frac{3\pi}2\nu_j\right)\right)\,\mathrm d \nu_1\mathrm d \nu_2,
\end{align*}
where $\nu=(\nu_1,\nu_2)\in\mathbb C^2$ and $\nu_3=-\nu_1-\nu_2$.

Let $h_T$ be the weight function defined in (\ref{weight-function}). Then we have the following \cite[Theorem 3.1.]{Buttcane-Zhou}.
\begin{lemma}\label{asymptotic}
For integers $\ell_1,\ell_2\geq 0$ we have
\begin{align*}
\sum_\phi A_\phi(p^{\ell_1},p^{\ell_2})h_T(\nu_\phi)=\mathfrak M_{\aleph(\ell_2,\ell_1)}^0(p^{-1})\left(\int_{\Re\nu=0}h_T(\nu)\text{spec}(\nu)\,\mathrm d\nu\right)+O\left(T^{14/3+\varepsilon}p^{(\ell_1+\ell_2)/2+\varepsilon}\right)
\end{align*}
for any $\varepsilon>0$.
\end{lemma}

\begin{remark}\label{number-of-forms}
It is known that 
\begin{align*}
\sum_\phi h_T(\nu_\phi)\asymp\int_{\Re\nu=0} h_T(\nu)\text{spec}(\nu)\,\mathrm d \nu\sim C\cdot T^{5-2\eta}
\end{align*}
for some absolute constant $C>0$. 
\end{remark}

\noindent The following result of Lau and Wang \cite{Lau-Wang} will be needed in the proof of Theorem \ref{Hecke-aa}. Let $X_T$ be the set of classical Hecke-Maass cusp forms with spectral parameter bounded by $T$. 
\begin{lemma}\label{Lau-Wang}
Let $p$ be a prime. Then for any interval $[\alpha,\beta]\subset [-2,2]$ we have
\begin{align*}
\frac{\#\left\{g\in X_T:\,\lambda_g(p)\in [\alpha,\beta]\right\}}{\#X_T}=\int\limits_\alpha^\beta\mathrm d \mu_p+O\left(\frac{\log p}{\log T}\right), 
\end{align*}
where $\mathrm d \mu_p$ is the $p$-adic Plancherel measure for the group $\mathrm{SL}_2(\mathbb C)$ (see \cite{Sarnak}). 
\end{lemma}

\begin{remark}
Analogous result was proved earlier by Murty and Sinha \cite{Murty-Sinha} for holomorphic cusp forms. 
\end{remark}

\noindent Recently Matz and Templier \cite{Matz-Templier} have proved equidistribution of the Satake parameters of higher rank Maass cusp forms with respect to the $p$-adic Plancherel measure. We formulate the following version of their result, which is used in the proof of Theorem \ref{Hecke-aa} to treat the forms violating the Ramanujan-Petersson conjecture at a prime $p$ (compare to \cite[Theorem 2.1.]{Lau-Ng-Wang-2020}). Let $\mathbb C[x_1^{\pm},x_2^\pm,x_3^\pm]^W$ be the ring of $W$-invariant Laurent polynomials with three variables.  

\begin{theorem}\label{Matz-Templier}
Let $h_T$ be the same weight function as above. Then, for any $f\in\mathbb C[x_1^{\pm},x_2^\pm,x_3^\pm]^W$ and prime $p$, we have
\begin{align*}
\left|\sum_{\phi}h_T(\nu_\phi)f(\alpha_\phi(p))-\left(\sum_\phi h_T(\nu_\phi)\right)\int_{T_0/W}f(\sigma)\,\mathrm d \mu_p(\sigma)\right|\ll T^{9/2-2\eta}p^{A\cdot\text{deg}'(f)}\|f\|_{\max},
\end{align*}
where the implied constant and the constant $A$ are absolute. Here $\|f\|_{\max}$ denotes the maximum of the absolute values of its coefficients and the constant term. The degree function $\text{deg}'(f)$ denotes the degree when $f$ is expressed in terms of the elementary symmetric
polynomials $e_0,...,e_3$ $(e_0:= 1$ and $e_3:= x_1\cdots x_3)$ with $\text{deg}'(e_0)=\text{deg}'(e_3)=0$ and $\text{deg}(e_i) = 1$ for $1\leq i\leq 2$.
\end{theorem}

\noindent On the analytic side of the argument we will rely on the mean value theorem for Dirichlet polynomials \cite[Chapter 9]{Iwaniec-Kowalski}.
\begin{lemma}\label{MVT}
Let $N\geq 1$ and $F(s):=\sum_{n\sim N}a_nn^{-s}$ , where $a_n$ are any complex numbers. Then
\begin{align*}
\int\limits_{-T}^T\left|F\left(\frac 12+it\right)\right|^2\,\mathrm d t\ll (N+T)\sum_{n\sim N}\frac{|a_n|^2}{n}.
\end{align*}
\end{lemma}

\noindent We will employ the following result of Matom\"aki and Radziwi{\l\l} \cite[Corollary 6]{Matomaki-Radziwill2} in several places.

\begin{lemma}\label{positprop}
Let $f:\mathbb N\longrightarrow \mathbb R$ be a multiplicative function. Then the sequence $\{f(n)\}$ has a positive
proportion of sign changes if and only if $f(n)<0$ for some integer $n>0$ and $f(n)\neq 0$ for a positive proportion of integers $n$.
\end{lemma}

\noindent Our last two theorems are based on Elliott-Montgomery-Vaughan type large sieve inequalities. We have the following result due to Xiao ja Xu \cite[Theorem 1]{Xiao-Xu}. Let $h_T$ be the same weight function as in (\ref{weight-function}). 

\begin{lemma}\label{large-sieve}
Suppose $e_1, e_2$ are two fixed non-negative integers, and $e_1+e_2\geq 1$.
Let $\{b_p\}_p$ be a sequence of complex numbers indexed by prime numbers such that $|b_p|\leq B$
for some constant $B$ and for all primes $p$. Then we have
\begin{align*}
&\sum_{\phi}h_T(\nu_{\phi})\left|\sum_{P<p\leq Q}\frac{b_p A_\phi(p^{e_1},p^{e_2})}p\right|^{2k}\\
&\ll_{\varepsilon,\nu_0} T^5\left(\frac{2^9(e_1+e_2)^8B^2k}{P\log P}\right)^k\left(1+\left(\frac{40k\log P}P\right)^{k/3}\right)+T^{14/3+\varepsilon}\left(\frac{8(e_1+e_2)^4BQ^{(e_1+e_2)/2+\varepsilon}}{\log P}\right)^{2k}
\end{align*}
uniformly for $B>0$, $k\geq 1$, and $2\leq P< Q\leq 2P$.  
\end{lemma}

\noindent Similar result also holds when restricting to self-dual forms. Let $\widetilde h_T$ is the weight function constructed by Guerreiro \cite{Guerreiro}.

\begin{lemma}\label{self-dual-large-sieve}
Suppose $e_1, e_2$ are two fixed non-negative integers, and $e_1+e_2\geq 1$.
Let $\{b_p\}_p$ be a sequence of complex numbers indexed by prime numbers such that $|b_p|\leq B$
for some constant $B$ and for all primes $p$. Then we have
\begin{align*}
&\sum_{\phi}\widetilde h_T(\nu_\phi)\left|\sum_{P<p\leq Q}\frac{b_p A_\phi(p^{e_1},p^{e_2})}p\right|^{2k}\\
&\ll_{\varepsilon,\nu_0} \frac{T^4}{\sqrt{\log T}}\left(\frac{2^9(e_1+e_2)^8B^2k}{P\log P}\right)^k\left(1+\left(\frac{40k\log P}P\right)^{k/3}\right)+T^{11/3+\varepsilon}\left(\frac{8(e_1+e_2)^4BQ^{(e_1+e_2)/2+\varepsilon}}{\log P}\right)^{2k}
\end{align*}
uniformly for $B>0$, $k\geq 1$, and $2\leq P< Q\leq 2P$. 
\end{lemma}

\begin{remark}
We have that $\widetilde h_T(\nu_\phi)\asymp 1$ when $\phi$ is self-dual and $\|\nu_\phi\|\leq T$. Note that Guerreiro's weight function depends on a scaling parameter $R$, which is needed for technical reasons, but one gets rid of this by renormalisation. Also, the weight function $\widetilde h_T$ depends on some chosen $\nu_0\in i\mathfrak a_{\mathbb C}^*$ similarly as the weight function $h_T$.
\end{remark}

\noindent This follows immediately inserting the estimate 
\begin{align*}
\sum_{\phi}\frac{\left|\widetilde h_T(\nu_\phi)\right|}{L(1,\text{Ad}^2\phi)}\ll\frac{T^4}{\sqrt{\log T}}
\end{align*}
established by Guerreiro \cite{Guerreiro} into the argument of Xiao and Xu. Here $L(s,\text{Ad}^2\phi)$ is the adjoint square $L$-function attached to $\phi$. 

\section{Proof of Theorem \ref{merkkivaihtelu}}

\noindent Let $X^{2\delta}\leq H\ll X$ for some arbitrarily small fixed $\delta>0$. We set $M=X^\delta$ so that $M<H\ll X$. Let us consider the sum
\begin{align*}
\sum_{\substack{x\leq mk\leq x+H\\
m\sim M\\
(m,k)=1}}A(mk,1).
\end{align*}

\noindent Theorem \ref{merkkivaihtelu} will follow from the following two lemmas. Recall that $\theta\geq 0$ is such that $L(1/2+it,\phi)\ll_\varepsilon (3+|t|)^{\theta+\varepsilon}$ for any $\varepsilon>0$. 

\begin{lemma}\label{lemma 1}
Assume that $\theta> 1/2$ and $H\geq X^{1-1/2\theta+10\delta}$. Then we have 
\begin{align}\label{main-sum}
\int\limits_X^{2X}\left|\sum_{\substack{x\leq mk\leq x+H\\
m\sim M\\
(m,k)=1}}A(mk,1)\right|^2\mathrm d x\ll H^2X^{1-\delta^2}.
\end{align}
The same estimate holds for any $H\geq X^{2\delta}$ when $0\leq \theta\leq 1/2$. 
\end{lemma}

\begin{lemma}\label{lemma 2}
Let $\varepsilon>0$ be small but fixed. Then we have 
\begin{align*}
\sum_{\substack{x\leq mk\leq x+H\\
m\sim M\\
(m,k)=1}}|A(mk,1)|>\frac H{X^{\varepsilon}}
\end{align*}
for $\gg_\varepsilon X^{1-3\varepsilon/2}$ of $x\sim X$. 
\end{lemma}
\noindent With these lemmas the proof of Theorem \ref{merkkivaihtelu} proceeds as follows. From Lemma \ref{lemma 1} we infer by Chebyshev's inequality that there exists a positive constant $C$ depending on $\delta$ so that
\begin{align*}
\left|\sum_{\substack{x\leq mk\leq x+H\\
m\sim M\\
(m,k)=1}}A(mk,1)\right|\leq CKHx^{-\delta^2/2}
\end{align*}
for at least a proportion $1-1/K^2$ of $x\sim X$ for appropriate $H$. Choosing, say, $\varepsilon=\delta^2/4$ in Lemma \ref{lemma 2} we see that 
\begin{align*}
\left|\sum_{\substack{x\leq mk\leq x+H\\
m\sim M\\
(m,k)=1}}A(mk,1)\right|<\sum_{\substack{x\leq mk\leq x+H\\
m\sim M\\
(m,k)=1}}|A(mk,1)|
\end{align*}
for $\gg X^{1-3\delta^2/8}$ of $x\sim X$. Hence, for $H\geq X^{1-1/2\theta+10\delta}$ (if $\theta> 1/2)$ or $H\geq X^{2\delta}$ (if $0\leq\theta\leq 1/2)$, the sequence $\{A(m,1)\}_m$ has a sign change in the interval $[x,x+H]$ for $\gg X^{1-3\delta^2/8}$ of $x\sim X$. From this we deduce that the number of sign changes up to $X$ is $\gg X^{1/2\theta-3\delta^2/8-10\delta}$ (if $\theta> 1/2)$ or $\gg X^{1-3\delta^2/8-10\delta}$ (if $0\leq\theta\leq 1/2)$, giving the claimed result. \qed

\subsection{Proof of Lemma \ref{lemma 1}} By Perron's formula
\begin{align*}
\sum_{\substack{x\leq mk\leq x+H\\
m\sim M\\
(m,k)=1}}A(mk,1)=\frac1{2\pi i}\int\limits_{(1+\Xi)}\sum_{\substack{m\sim M\\
X/3M\leq k\leq 3X/M\\
(m,k)=1}}\frac{A(mk,1)}{(mk)^s}\cdot\frac{(x+H)^s-x^s}s\,\mathrm d s
\end{align*}
for any $\Xi>0$. Using an idea of Saffari and Vaughan \cite{S-V} we write
\begin{align*}
\frac{(x+H)^s-x^s}s=\frac x{2H}\int\limits_{H/x}^{3H/x}x^s\frac{(1+u)^s-1}s\,\mathrm d u-\frac{x+H}{2H}\int\limits_0^{2H/(x+H)}(x+H)^s\frac{(1+u)^s-1}s\,\mathrm d u. 
\end{align*}
It is enough to consider the contribution of the first term on the right-hand side as the other term can be treated similarly. Hence, 
\begin{align*}
\int\limits_X^{2X}\left|\sum_{\substack{x\leq mk\leq x+H\\
m\sim M\\
(m,k)=1}}A(mk,1)\right|^2\mathrm d x&\ll \frac{X^2}{H^2}\int\limits_X^{2X}\left|\int\limits_{H/x}^{3H/x}\int\limits_{(1+\Xi)}\sum_{\substack{m\sim M\\
X/3M\leq k\leq 3X/M\\
(m,k)=1}}\frac{A(mk,1)}{(mk)^s}x^s\frac{(1+u)^s-1}s\,\mathrm d s\mathrm d u\right|^2\,\mathrm d x\\
&\ll\frac XH\int\limits_{H/2X}^{3H/X}\int\limits_X^{2X}\left|\int\limits_{(1+\Xi)}\sum_{\substack{m\sim M\\
X/3M\leq k\leq 3X/M\\
(m,k)=1}}\frac{A(mk,1)}{(mk)^s}x^s\frac{(1+u)^s-1}s\,\mathrm d s\right|^2\,\mathrm d x\mathrm d u\\
&\ll\int\limits_X^{2X}\left|\int\limits_{(1+\Xi)}\sum_{\substack{m\sim M\\
X/3M\leq k\leq 3X/M\\
(m,k)=1}}\frac{A(mk,1)}{(mk)^s}x^s\frac{(1+u)^s-1}s\,\mathrm d s\right|^2\,\mathrm d x
\end{align*}
for some $u\ll H/X$. 

We detect the condition $(k,m)=1$ by M\"obius inversion, which enables us to write the integral above as 
\begin{align*}
\int\limits_X^{2X}\left|\int\limits_{(1+\Xi)}\sum_{d\leq 2M}\mu(d)\sum_{\substack{m\sim M\\
m\equiv 0\,(d)}}\frac{A(m,1)}{m^s}\sum_{\substack{X/3M\leq k\leq 3X/M\\
k\equiv 0\,(d)}}\frac{A(k,1)}{k^s}x^s\frac{(1+u)^s-1}s\mathrm d s\right|^2\mathrm d x
\end{align*}
using the multiplicativity of Hecke eigenvalues. Due to the presence of the M\"obius function we can restrict to squarefree $d$. In this case comparing Euler products gives
\begin{align*}
\sum_{\substack{m\sim M\\
m\equiv 0\,(d)}}\frac{A(m,1)}{m^s}=\frac1{d^s}\prod_{p|d}\frac{\sum_{j=0}^\infty\frac{A(p^{j+1},1)}{p^{js}}}{\sum_{j=0}^\infty\frac{A(p^j,1)}{p^{js}}}\sum_{m\sim M}\frac{A(m,1)}{m^s}
\end{align*}
and similar factorisation holds also for the other Dirichlet polynomial. 

It is well-known \cite[p. 174]{Goldfeld} that
\begin{align*}
\sum_{j=0}^\infty\frac{A(p^j,1)}{p^{js}}=\left(1-\frac{A(p,1)}{p^{s}}+\frac{A(p,1)}{p^{2s}}-\frac1{p^{3s}}\right)^{-1}
\end{align*}
for $\sigma>1$. Hence, in the same region we have
\begin{align*}
\prod_{p|d}\frac{\sum_{j=0}^\infty\frac{A(p^{j+1},1)}{p^{js}}}{\sum_{j=0}^\infty\frac{A(p^j,1)}{p^{js}}}&=\prod_{p|d}\left(A(p,1)-\frac{A(p,1)}{p^s}+\frac1{p^{2s}}\right).
\end{align*}
From now on we set
\begin{align*}
M(s):=\sum_{m\sim M}\frac{A(m,1)}{m^s},\quad\quad K(s):=\sum_{X/3M\leq k\leq 3X/M}\frac{A(k,1)}{k^s},
\end{align*}
and
\begin{align*}
D(s):=\sum_{d\leq 2M}\mu(d)\prod_{p|d}\left(\frac{A(p,1)}{p^s}-\frac{A(p,1)}{p^{2s}}+\frac1{p^{3s}}\right)^2.
\end{align*}
We note that 
\begin{align}\label{d-estimate}
D\left(\sigma+it\right)\ll \sum_{d\leq 2M}\frac{|A(d,1)|^2}{d^{2\sigma}}\ll M^{1-2\sigma}\log M
\end{align}
for $1/2\leq\sigma\leq 1+\Xi$.

At this point our integral is given by
\begin{align*}
\int\limits_X^{2X}\left|\int\limits_{(1+\Xi)}M(s)K(s)D(s)x^s\frac{(1+u)^s-1}s\mathrm d s\right|^2\mathrm d x.
\end{align*}
Let us first shift the line of integration in the $s$-integral from the line $\sigma=1+\Xi$ to the line $\sigma=1/2$. Let $U\gg X/H$ be a large enough parameter and $g$ be smooth weight function supported in the interval $[1/2,4]$, which is identically one in the interval $[1,2]$. Then the integral along the half-line connecting $1+\Xi\pm iU$ to $1+\Xi\pm i\infty$ contributes 
\begin{align*}
&\ll\int\limits_{\mathbb R}g\left(\frac xX\right)\left|\int\limits_{1+\Xi\pm iU}^{1+\Xi\pm i\infty}M(s)K(s)D(s)x^s\frac{(1+u)^s-1}s\mathrm d s\right|^2\mathrm d x\\
&\ll\int\limits_{1+\Xi\pm iU}^{1+\Xi\pm i\infty}\int\limits_{1+\Xi\pm iU}^{1+\Xi\pm i\infty}\left|M(s_1)M(s_2)K(s_1)K(s_2)D(s_1)D(s_2)\frac{(1+u)^{s-1}-1}{s_1}\frac{(1+u)^{s_2}-1}{s_2}\right|\\
&\qquad\qquad\cdot\left|\int\limits_{\mathbb R}g\left(\frac xX\right)x^{s_1+\overline{s_2}}\mathrm d x\right|\mathrm d s_1 \mathrm d s_2\\
&\ll X^{3+2\Xi}\int\limits_{1+\Xi\pm iU}^{1+\Xi\pm i\infty}\frac{|M(s)K(s)D(s)|^2}{|t|^2}\mathrm d s\\
&\ll \frac{X^{3+2\Xi}}{U},
\end{align*}
where we have used the estimate 
\begin{align*}
\frac{(1+u)^s-1}s\ll\min\left\{\frac HX,\frac1{|t|}\right\}
\end{align*} in the penultimate step (recall that $U\gg X/H$) and the fact that $M(s)K(s)D(s)$ is bounded on the line $\sigma=1+\Xi$ in the last step. 

On the horizontal lines $[1/2\pm iU,1+\Xi\pm iU]$ the integral is similarly bounded by
\begin{align*}
&\ll X\int\limits_{1/2}^{1+\Xi}\left|M(\sigma\pm iU)K(\sigma\pm iU)D(\sigma\pm iU)\frac{(1+u)^{\sigma\pm iU}-1}{\sigma\pm iU}\right|^2\,\mathrm d \sigma\\
&\ll\frac X{U^{1/2}}
\end{align*}
using the convexity bound $M(s)K(s)\ll(1+|t|)^{3(1-\sigma)/2}$ in the strip $1/2\leq\sigma\leq 1+\Xi$, which follows from the functional equation of $L(s,\phi)$ and the Phragm\'en-Lindel\"of principle. Choosing $U$ to be sufficiently large in terms of $X$ and $H$ we see that the contributions from the shift can be made arbitrarily small.  


Now rest of the integral is
\begin{align*}
&\ll\int\limits_\mathbb R g\left(\frac xX\right)\left|\int\limits_{1/2}^{1/2+iU}M(s)K(s)D(s)x^s\frac{(1+u)^s-1}s\mathrm d s\right|^2\mathrm d x\\
&\ll\int\limits_{1/2}^{1/2+iU}\int\limits_{1/2}^{1/2+iU}\left|M(s_1)M(s_2)K(s_1)K(s_2)D(s_1)D(s_2)\frac{(1+u)^{s_1}-1}{s_1}\frac{(1+u)^{s_2}-1}{s_2}\right|\\
&\qquad\qquad\quad\cdot\left|\int\limits_\mathbb R g\left(\frac xX\right)x^{s_1+\overline{s_2}}\mathrm d x\right|\mathrm d s_1\mathrm d s_2\\
&\ll X^2\int\limits_{1/2}^{1/2+iU}|M(s)K(s)D(s)|^2\min\left\{\frac{H^2}{X^2},\frac1{|t|^2}\right\}\,\mathrm d s\\
&\ll H^2\int\limits_{1/2}^{1/2+iX/H}\left|M(s)K(s)D(s)\right|^2\,\mathrm d s+X^2\int\limits_{1/2+iX/H}^{1/2+iU}\frac{\left|M(s)K(s)D(s)\right|^2}{|t|^2}\,\mathrm d s.
\end{align*}
Hence, the goal is to estimate the integral
\begin{align*}
\int\limits_0^{X/H}\left|M\left(\frac12+it\right)K\left(\frac12+it\right)D\left(\frac12+it\right)\right|^2\,\mathrm d t.
\end{align*}
We set
\begin{align*}
\mathcal T:=\left\{t\in[0,X/H]:\,\left|M\left(\frac12+it\right)\right|\leq M^{1/2-\delta}\right\}
\end{align*}
and $\mathcal U:=[0,X/H]\backslash\mathcal T$. Using the definition of $\mathcal T$ and Lemma \ref{MVT} for the $K$-polynomial we get
\begin{align*}
&\int\limits_{\mathcal T}\left|M\left(\frac12+it\right)K\left(\frac12+it\right)D\left(\frac12+it\right)\right|^2\,\mathrm d t\\
&\ll M^{1-2\delta}\left(\frac XH+\frac XM\right)\sum_{m\sim X/M}\frac{|A(m,1)|^2}m(\log M)^2\ll XM^{-2\delta}(\log X)^3\ll X^{1-\delta^2}
\end{align*}
using the Rankin-Selberg theory. 

Let $r$ be the smallest natural number such that $M^r\geq X/H$. Using Lemma \ref{MVT} for the polynomial $M(1/2+t)^r$ and applying the Cauchy-Schwarz inequality yields
\begin{align*}
|\mathcal U|M^{2r(1/2-\delta)}&\leq\int\limits_0^{X/H}\left|M\left(\frac12+it\right)\right|^{2r}\mathrm d t\\
&\ll\left(\frac XH+M^r\right)\sum_{n\asymp M^r}\left(\sum_{\substack{n=m_1\cdots m_r\\
m_i\sim M}}\frac{|A(m_1,1)\cdots A(m_r,1)|}{(m_1\cdots m_r)^{1/2}}\right)^2\\
&\ll M^r\sum_{n\asymp M^r}d_r(n)\left(\sum_{\substack{n=m_1\cdots m_r\\
m_i\sim M}}\frac{|A(m_1,1)\cdots A(m_r,1)|^2}{m_1\cdots m_r}\right)\\
&\ll M^rX^{\delta^2}\frac1{M^r}\left(\sum_{m\sim M}|A(m,1)|^2\right)^r\\
&\ll M^rX^{\delta^2},
\end{align*}
from which we infer that $|\mathcal U|\ll X^{4\delta}$.

Thus, estimating the $M$-polynomial trivially and the $D$-polynomial by (\ref{d-estimate}) yields
\begin{align*}
\int\limits_{\mathcal U}\left|M\left(\frac12+it\right)K\left(\frac12+it\right)D\left(\frac12+it\right)\right|^2\mathrm d t\ll X^{7\delta}\sup_{|t|\leq X/H}\left|K\left(\frac12+it\right)\right|^2.
\end{align*}
Combining all the estimates above gives
\begin{align*}
&\int\limits_X^{2X}\left|\sum_{\substack{x\leq mk\leq x+H\\
m\sim M\\
(m,k)=1}} A(mk,1)\right|^2\mathrm d x\\
&\ll H^2\left(X^{1-\delta^2}+X^{7\delta}\sup_{|t|\leq X/H}\left|K\left(\frac12+it\right)\right|^2\right)\\
&\qquad+X^2\sup_{\frac XH\leq T\leq U}\frac1{T^2}\int\limits_T^{2T}\left|M\left(\frac 12+it\right)K\left(\frac12+it\right)D\left(\frac12+it\right)\right|^2\,\mathrm d t.
\end{align*}
Using the bound $K(1/2+it)\ll_\varepsilon (3+|t|)^{\theta+\varepsilon}$ (which follows from the same bound for $L(1/2+it)$) the first two terms on the right-hand side are both $\ll H^2X^{1-\delta^2}$ for $H\geq X^{1-1/2\theta+10\delta}$ when $\theta>1/2$ (and respectively for $H\geq X^{2\delta}$ when $0\leq\theta\leq 1/2)$. The penultimate term can be treated by Lemma \ref{MVT} similarly as before by splitting the range of integration into parts:
\begin{align*}
&X^2\sup_{\frac XH\leq T\leq U}\frac1{T^2}\int\limits_T^{2T}\left|M\left(\frac 12+it\right)K\left(\frac12+it\right)D\left(\frac12+it\right)\right|^2\,\mathrm d t\\
&\ll X^2\sup_{X/H\leq T\leq U} \frac{M^{1-2\delta}(\log X)^3}T+X^2\sup_{X/H\leq T\leq U}\frac{XM^{-2\delta}(\log X)^3}{T^2}+X^2\sup_{X/H\leq T\leq U}\frac{X^{7\delta}T^{6/5}(\log X)^2}{T^2}\\
&\ll HX^{1+\delta-2\delta^2}(\log X)^3+X^{1-2\delta^2}H^2(\log X)^3+X^{6/5+7\delta}H^{4/5}\\
&\ll H^2X^{1-\delta^2}
\end{align*}
for $H$ in the same range as before. This finishes the proof. \qed


\subsection{Proof of Lemma \ref{lemma 2}}

Let $\varepsilon>0$ be small but fixed. Observe that the proof of Theorem 4.1. in \cite{JiangLu} gives\footnote{Notice that there is a gap in the proof of \cite[Theorem 4.1.]{JiangLu} as the asymptotics (4.3) does not follow from the display above it as claimed. However, (4.3) follows in the case of self-dual form from the zero-free region of the corresponding L-function and the method of de la Vall\'ee-Poussin \cite[Chapter 5]{Iwaniec-Kowalski}, see e.g. \cite[Corollary 1.2.]{Liu-Wang-Ye}} the bound
\begin{align}\label{Jiang-Lu-estimate}
\sum_{\ell\ll\log X}\sum_{p^\ell\sim X}|A(p^\ell,1)|\gg X^{1-\varepsilon/2}.
\end{align}
Note that 
\begin{align}\label{separation}
\sum_{\substack{mk\sim X\\
m\sim M\\
(m,k)=1}}|A(mk,1)|&=\sum_{m\sim M}|A(m,1)|\sum_{\substack{k\sim X/m\\
(m,k)=1}}|A(k,1)|
\end{align}
by the multiplicativity of Hecke eigenvalues. Let us treat the inner sum first. By restricting to prime powers we have
\begin{align*}
\sum_{\substack{k\sim X/m\\
(m,k)=1}}|A(k,1)|&\geq\sum_{\ell\ll\log(X/M)}\sum_{\substack{p^\ell\sim X/M\\
p\nmid m}}|A(p^\ell,1)|\\
&=\sum_{\ell\ll\log(X/M)}\sum_{p^\ell\sim X/M}|A(p^\ell,1)|-\sum_{\ell\ll\log(X/M)}\sum_{\substack{p^\ell\sim X/M\\
p|m}}|A(p^\ell,1)|.
\end{align*}
The first term on the right-hand side is $\gg (X/M)^{1-\varepsilon/2}$ by (\ref{Jiang-Lu-estimate}). The other term on the right-hand side is
\begin{align*}
\ll\left(\frac XM\right)^{1/2}\sum_{\ell\ll\log X}\sum_{\substack{p^\ell\sim X/M\\
p|m}}1\ll \left(\frac XM\right)^{1/2+\varepsilon}
\end{align*}
by estimating the Hecke eigenvalues pointwise.

We conclude that 
\begin{align*}
\sum_{\substack{k\sim X/m\\
(m,k)=1}}|A(k,1)|\gg\left(\frac XM\right)^{1-\varepsilon/2}
\end{align*}
uniformly for $m\sim M$.

This and (\ref{separation}) yield
\begin{align}\label{long-lower-bound}
\sum_{\substack{mk\sim X\\
m\sim M\\
(m,k)=1}}|A(mk,1)|&\gg \left(\frac XM\right)^{1-\varepsilon/2}\sum_{\ell\ll\log M}\sum_{p^\ell\sim M}|A(p^\ell,1)|\gg X^{1-\varepsilon/2}.
\end{align}
By the Rankin-Selberg theory we have
\begin{align*}
\sum_{\substack{mk\sim X\\
m\sim M\\
(m,k)=1\\
|A(mk,1)|> X^{\varepsilon}}}|A(mk,1)|< X^{-\varepsilon}\sum_{\substack{mk\sim X\\
m\sim M\\
(m,k)=1}}|A(mk,1)|^2\ll X^{1-\varepsilon}.
\end{align*}
Thus 
\begin{align*}
\sum_{\substack{mk\sim X\\
m\sim M\\
(m,k)=1\\
|A(mk,1)|\leq X^{\varepsilon}}}|A(mk,1)|\gg X^{1-\varepsilon/2}.
\end{align*}
We conclude that
\begin{align}\label{long-short-estimate}
X^{1-\varepsilon/2}\ll \sum_{\substack{mk\sim X\\
m\sim M\\
(m,k)=1\\
|A(mk,1)|\leq X^{\varepsilon}}}|A(mk,1)|\leq\frac 1H\int\limits_X^{2X}\sum_{\substack{x\leq mk\leq x+H\\
m\sim M\\
(m,k)=1}}
|A(mk,1)|\,\mathrm d x.
\end{align}
Let us define
\begin{align*}
\mathcal S:=\left\{x\sim X:\, \sum_{\substack{x\leq mk\leq x+H\\
m\sim M\\
(m,k)=1}}|A(mk,1)|\leq \frac H{X^{\varepsilon}}\right\}.
\end{align*}
Points $x\in \mathcal S$ contribute 
\begin{align*}
\ll \frac 1H\cdot X\cdot\frac H{X^{\varepsilon}}\ll X^{1-\varepsilon}
\end{align*}
to the right-hand side of (\ref{long-short-estimate}). This implies that 
\begin{align*}
X^{1-\varepsilon/2}\ll \frac 1H|\mathcal S^c|\cdot HX^{\varepsilon},
\end{align*}
i.e.
\begin{align*}
|\mathcal S^c|\gg X^{1-3\varepsilon/2}.
\end{align*}
For this many $x\sim X$ we have
\begin{align*}
\sum_{\substack{x\leq mk\leq x+H\\
m\sim M\\(m,k)=1}}|A(mk,1)|>\frac H{X^{\varepsilon}}, 
\end{align*}
which concludes the proof.  \qed

\section{Proof of Theorem \ref{Hecke-aa}}

It follows from part (i) of Theorem \ref{merkkivaihtelu} that there is a $m$ so that $A(m,1)<0$ and so by Lemma \ref{positprop} it suffices to show that for all but $\varepsilon\#\mathcal H_T$ forms $\phi\in\mathcal H_T$ we have $A_\phi(m,1)\neq 0 $ for a positive proportion of integers $m$. By \cite[Th\'eor\`eme 14]{Serre} this holds if 
\begin{align*}
\sum_{\substack{p\\
A_\phi(p,1)=0}}\frac1p<\infty.
\end{align*}
Recall that self-dual Maass cusp forms for the group $\mathrm{SL}_3(\mathbb Z)$ are symmetric square lifts of non-dihedral cusp forms for the group $\mathrm{SL}_2(\mathbb Z)$. Furthermore, if a $\mathrm{GL}_3$ form $\phi$ is a lift of a $\mathrm{GL}_2$ form $g$, then the Fourier coefficients at primes are related by $A_\phi(p,1)=\lambda_g(p)^2-1$. It follows immediately from Lemma \ref{Lau-Wang} that 
\begin{align*}
\frac{\#\{g\in X_T:\,\sqrt{1-p^{-\delta}}<|\lambda_g(p)|<\sqrt{1+p^{-\delta}}\}}{\# X_T}\ll \frac 1{p^{\delta}}+\frac{\log p}{\log T}
\end{align*}
for any $\delta>0$, and consequently 
\begin{align*}
\frac{\#\{\phi\in\mathcal H_T:\,|A_\phi(p,1)|<p^{-\delta}\}}{\#\mathcal H_T}\ll \frac 1{p^{\delta}}+\frac{\log p}{\log T}
\end{align*}
as dihedral forms have zero density among the classical Maass cusp forms.

This gives
\begin{align*}
\sum_{\phi\in\mathcal H_T}\sum_{\substack{p\leq X\\
A_\phi(p,1)=0}}\frac1p\ll \#\mathcal H_T\sum_{p\leq X}\frac1p\left(\frac1{p^\delta}+\frac{\log p}{\log T}\right)\ll \#\mathcal H_T
\end{align*}
for any $X>1$ such that $\log X\ll\log T$ when $T$ is sufficiently large. This implies that there is an absolute constant $C>0$ such that, for all but at most $\varepsilon\#\mathcal H_T$ forms in $\mathcal H_T$, we have
\begin{align*}
\sum_{\substack{p\leq X\\
A_\phi(p,1)=0}}\frac 1p\leq\frac C{2\varepsilon}.
\end{align*} 
Letting $T\longrightarrow\infty$ implies the result by the above discussion if $X$ is chosen so that $X\longrightarrow\infty$ along with $T$. \qed
 
\section{Proof of Theorem \ref{aa}}

We will need the following effective Sato-Tate result for the coefficients $A(m,m)$. 
\begin{theorem}\label{Sato-Tate}
Let $p$ be a prime so that $\log p\ll \log T$ and $[\alpha,\beta]\subset[-1,8]$. Let $h_T(\nu)$ be the weight function defined in (\ref{weight-function}). Then we have
\begin{align*}
\frac1{\int_{\Re\nu=0}h_T(\nu)\text{spec}(\nu)\mathrm d \nu}\sum_{\substack{\phi\\
A_\phi(p,p)\in[\alpha,\beta]}}h_T(\nu_\phi)=\int_{T_0/W}1_{[\alpha,\beta]}(S_{1,1}(\sigma))\,\mathrm d \mu_p(\sigma)+O\left(\left(\frac{\log p}{\log T}\right)^{1/5}\right),
\end{align*}
where $S_{1,1}(\sigma)$ is the Schur polynomial defined in (\ref{Schur}) and $\mathrm d \mu_p$ is the $p$-adic Plancherel measure for the group $\mathrm{SL}_3(\mathbb C)$ given in (\ref{p-adic-plancherel}). 
\end{theorem}
\begin{remark}
An easy computation shows that $S_{1,1}(\alpha_1,\alpha_2,\alpha_3)=(\alpha_1+\alpha_2)(\alpha_2+\alpha_3)(\alpha_3+\alpha_1)$ for $(\alpha_1,\alpha_2,\alpha_3)\in T_0$ and that $A_\phi(p,p)=(\alpha_{1,p}(\phi)+\alpha_{2,p}(\phi))(\alpha_{2,p}(\phi)+\alpha_{3,p}(\phi))(\alpha_{3,p}(\phi)+\alpha_{1,p}(\phi))$.
\end{remark}
\begin{proof} 



Let $0<\delta<(\beta-\alpha)/2$ be a small parameter optimised later. Let $w_1$ be a non-negative weight function, which is supported in the interval $[(\alpha-\delta+1)/9,(\beta+\delta+1)/9]$, identically one in $[(\alpha+1)/9,(\beta+1)/9]$, and satisfies $\|w_1'\|_\infty\ll 1/\delta$. It is well-known that such smooth function can be approximated by the means of Bernstein polynomials \cite{Bernstein}. Indeed, we have
\begin{align}\label{Bernstein-approx}
w_1\left(x\right)=\sum_{j=0}^n w_1\left(\frac j{n}\right)\binom nj x^j(1-x)^{n-j}+O\left(\frac{\|w_1\|_\infty^{1/3}}{n^{1/3}\delta^{2/3}}\right)
\end{align}
for any $x\in[0,1]$, where $n$ is a natural number specified later.

For simplicity let us set
\begin{align*}
\widetilde A_\phi(p,p):=\frac{A_\phi(p,p)+1}9.
\end{align*}
Observe that under the Ramanujan-Petersson conjecture we have $\widetilde A_\phi(p,p)\in[0,1]$.
 
Note also that
\begin{align}\label{weight-majorization}
1_{[\alpha,\beta]}(A_\phi(p,p))\leq w_1\left(\widetilde A_\phi(p,p)\right).
\end{align}
It follows from the Hecke relations that 
\begin{align*}
\left(\widetilde A_\phi(p,p)\right)^\ell&=\frac1{9^\ell}|A_\phi(p,1)|^{2\ell}\\
&=\sum_{\ell_1+\ell_2\leq 2\ell}\alpha_{\ell_1,\ell_2,\ell} A_\phi(p^{\ell_1},p^{\ell_2})
\end{align*}
for some real numbers $\alpha_{\ell_1,\ell_2,\ell}$. By induction we have
\begin{align}\label{induction-estimate}
\sum_{\ell_1+\ell_2\leq 2\ell}|\alpha_{\ell_1,\ell_2,\ell}|\leq 1,
\end{align} 
see \cite[Estimate (14)]{Blomer-Buttcane-Raulf}.

Hence
\begin{align}\label{upper-bound-january}
&\frac1{\int_{\Re\nu=0}h_T(\nu)\text{spec}(\nu)\mathrm d \nu}\sum_{\substack{\phi\\
A_\phi(p,p)\in[\alpha,\beta]}}h_T(\nu_\phi)\nonumber\\
&\leq\frac1{\int_{\Re\nu=0}h_T(\nu)\text{spec}(\nu)\mathrm d \nu}\sum_\phi w_1\left(\widetilde A_\phi(p,p)\right)h_T(\nu_\phi)\nonumber\\
&=\frac1{\int_{\Re\nu=0}h_T(\nu)\text{spec}(\nu)\mathrm d \nu}\sum_\phi\sum_{j=0}^n w_1\left(\frac j{n}\right)\binom nj \widetilde A_\phi(p,p)^j(1-\widetilde A_\phi(p,p))^{n-j}h_T(\nu_\phi)\nonumber\\
&\qquad\qquad\qquad+O(n^{-1/3}\delta^{-2/3})+O\left(\left(\frac{\log p}{\log T}\right)^{3/2}\right)\nonumber\\
&=\frac1{\int_{\Re\nu=0}h_T(\nu)\text{spec}(\nu)\mathrm d \nu}\sum_{j=0}^n w_1\left(\frac j{n}\right)\binom nj\sum_\phi\sum_{k=0}^{n-j}\binom{n-j}k (-1)^k \widetilde A_\phi(p,p)^{n-k}h_T(\nu_\phi)\nonumber\\
&\qquad\qquad\qquad+O(n^{-1/3}\delta^{-2/3})+O\left(\left(\frac{\log p}{\log T}\right)^{3/2}\right)\nonumber\\
&=\frac1{\int_{\Re\nu=0}h_T(\nu)\text{spec}(\nu)\mathrm d \nu}\sum_{j=0}^n w_1\left(\frac j{n}\right)\binom nj\sum_{k=0}^{n-j}\binom{n-j}k (-1)^k\nonumber\\
&\qquad\qquad\quad\sum_{\ell_1+\ell_2\leq 2(n-k)}\alpha_{\ell_1,\ell_2,n-k}\sum_\phi A_\phi\left(p^{\ell_1},p^{\ell_2}\right)h_T(\nu_\phi)+O(n^{-1/3}\delta^{-2/3})+O\left(\left(\frac{\log p}{\log T}\right)^{3/2}\right).
\end{align}

\noindent In the second step we use (\ref{Bernstein-approx}) for forms satisfying the Ramanujan-Petersson conjecture at prime $p$, but the treatment of the rest of the forms requires additional explanation. Let us define the polynomial $f:\mathbb C^3\longrightarrow\mathbb C$ by
\begin{align*}
&f(x_1,x_2,x_3)\\
&\quad:=\sum_{j=0}^n w_1\left(\frac j{n}\right)\binom nj((x_1+x_2)(x_2+x_3)(x_3+x_1))^j(1-(x_1+x_2)(x_2+x_3)(x_3+x_1))^{n-j}.
\end{align*}
Clearly $f\in\mathbb C[x_1^\pm,x_2^\pm,x_3^\pm]^W$. Note that this is well-defined regardless whether or not $x_1,x_2,x_3\in S^1$.

Recall that $X_p$ is the set of Hecke-Maass cusp forms for the group $\mathrm{SL}_3(\mathbb Z)$ not satisfying the Ramanujan-Petersson conjecture at $p$. We have to estimate the sums
\begin{align*}
\frac1{\int_{\Re\nu=0}h_T(\nu)\text{spec}(\nu)\mathrm d \nu}\sum_{\phi\in X_p}h_T(\nu_\phi)\quad\text{and}\quad\frac1{\int_{\Re\nu=0}h_T(\nu)\text{spec}(\nu)\mathrm d \nu}\sum_{\phi\in X_p}h_T(\nu_\phi)|f(\alpha_\phi(p))|.
\end{align*}
The first sum is $\ll (\log p/\log T)^{3}$ by (\ref{exceptions}). By the Cauchy-Schwarz inequality the second sum is 
\begin{align*}
&\ll\left(\frac1{\int_{\Re\nu=0}h_T(\nu)\text{spec}(\nu)\mathrm d \nu}\sum_{\phi\in X_p}h_T(\nu_\phi)\right)^{1/2}\left(\frac1{\int_{\Re\nu=0}h_T(\nu)\text{spec}(\nu)\mathrm d \nu}\sum_{\phi}h_T(\nu_\phi)|f(\alpha_\phi(p))|^2\right)^{1/2}.
\end{align*}
The second factor on the right-hand side is
\begin{align*}
&\ll\left(\int_{T_0/W}|f(\sigma)|^2\,\mathrm d\mu_p(\sigma)+O\left(T^{-1/2}\left(4p^{A}\right)^{2n}\right)\right)^{1/2}\\
&\ll 1+T^{-1/4}p^{An}\\
&\ll 1
\end{align*}
using Theorem \ref{Matz-Templier} and the choice of $n$ we make below, justifying the second step in (\ref{upper-bound-january}).  

Returning to (\ref{upper-bound-january}), by Lemma \ref{asymptotic} the innermost sum is
\begin{align}\label{spectral-sum}
\mathfrak M_{\aleph(\ell_2,\ell_1)}^0(p^{-1})\left(\int_{\Re(\nu)=0}h_T(\nu)\text{spec}(\nu)\,\mathrm d\nu\right)+O_\varepsilon\left(T^{14/3+\varepsilon}p^{(\ell_1+\ell_2)/2+\varepsilon}\right).
\end{align}
Let us first concentrate on the contribution of the main term. The latter factor in that term is cancelled by the denominator in (\ref{upper-bound-january}). By using Kato's formula (\ref{Kato}) the contribution from the main term
\begin{align*}
&\sum_{j=0}^n w_1\left(\frac j{n}\right)\binom nj\sum_{k=0}^{n-j}\binom{n-j}k (-1)^k\\
&\qquad\qquad\sum_{\ell_1+\ell_2\leq 2(n-k)}\alpha_{\ell_1,\ell_2,n-k}\int_{T_0/W}S_{\ell_1,\ell_2}(\sigma)\,\mathrm d \mu_p(\sigma)+O(n^{-1/3}\delta^{-2/3})+O\left(\left(\frac{\log p}{\log T}\right)^{3/2}\right).
\end{align*} 
By the relation $A_\phi(p^{\ell_1},p^{\ell_2})=S_{\ell_1,\ell_2}(\alpha_{1,p}(\phi),\alpha_{2,p}(\phi),\alpha_{3,p}(\phi))$ it follows that 
\begin{align*}
\left(\widetilde S_{1,1}(\sigma)\right)^\ell=\sum_{\ell_1+\ell_2\leq2\ell}\alpha_{\ell_1,\ell_2,\ell} S_{\ell_1,\ell_2}(\sigma)
\end{align*}
for any $\sigma\in T_0$, where we set
\begin{align*}
\widetilde S_{1,1}(\sigma):=\frac{S_{1,1}(\sigma)+1}9.
\end{align*}
Note that $\widetilde S_{1,1}(\sigma)\in[0,1]$ for $\sigma\in T_0$.

Therefore the main term above is 
\begin{align*}
&\int_{T_0/W}\sum_{j=0}^n w_1\left(\frac j{n}\right)\binom nj \widetilde S_{1,1}(\sigma)^j(1-\widetilde S_{1,1}(\sigma))^{n-j}\,\mathrm d \mu_p(\sigma)\\
&=\int_{T_0/W}w_1\left(\widetilde S_{1,1}(\sigma)\right)\,\mathrm d \mu_p(\sigma)+O\left(n^{-1/3}\delta^{-2/3}\right)+O\left(\left(\frac{\log p}{\log T}\right)^{3/2}\right).
\end{align*}
The contribution of the error term in (\ref{spectral-sum}) is 
\begin{align*}
&\ll\frac1{\int_{\Re\nu=0}h_T(\nu)\text{spec}(\nu)\mathrm d \nu}\sum_{j=0}^n w_1\left(\frac jn\right)\binom nj\sum_{k=0}^{n-j}(-1)^{n-j-k}\binom{n-j}k\sum_{\ell_1+\ell_2\leq 2(n-k)}\alpha_{\ell_1,\ell_2,n-k}T^{14/3+\varepsilon}p^{(\ell_1+\ell_2)/2+\varepsilon}\\
&\ll \frac{T^{14/3+\varepsilon}}{\int_{\Re\nu=0}h_T(\nu)\text{spec}(\nu)\mathrm d \nu}\sum_{j=0}^n w_1\left(\frac jn\right)\binom nj\sum_{k=0}^{n-j}\binom{n-j}k p^{(n-k)}\sum_{\ell_1+\ell_2\leq 2(n-k)}|\alpha_{\ell_1,\ell_2,n-k}|\\
&\ll \frac{T^{14/3+\varepsilon}}{\int_{\Re\nu=0}h_T(\nu)\text{spec}(\nu)\mathrm d \nu}\sum_{j=0}^n w_1\left(\frac jn\right)\binom nj p^{j}(1+p)^{n-j}\sum_{\ell_1+\ell_2\leq 2(n-k)}|\alpha_{\ell_1,\ell_2,n-k}|\\
&\ll \frac{T^{14/3+\varepsilon}}{\int_{\Re\nu=0}h_T(\nu)\text{spec}(\nu)\mathrm d \nu} (2p)^{n}\\
&\ll \frac{(2p)^n}{T^{1/3-\eta'}}
\end{align*}
for some very small $\eta'>0$, where in the penultimate step we have used the estimate (\ref{induction-estimate}).

Thus we have showed that 
\begin{align*}
&\frac1{\int_{\Re\nu=0}h_T(\nu)\text{spec}(\nu)\mathrm d \nu}\sum_{\substack{\phi\\
A_\phi(p,p)\in[\alpha,\beta]}}h_T(\nu_\phi)\\
&\qquad\leq \int_{T_0/W}w_1\left(\widetilde S_{1,1}(\sigma)\right)\,\mathrm d \mu_p(\sigma)+O\left(n^{-1/3}\delta^{-2/3}+\frac{(2p)^{n}}{T^{1/3-\eta'}}+\left(\frac{\log p}{\log T}\right)^{3/2}\right).
\end{align*}
Similar computation shows that  
\begin{align*}
&\int_{T_0/W}w_2\left(\widetilde S_{1,1}(\sigma)\right)\,\mathrm d \mu_p(\sigma)+O\left(n^{-1/3}\delta^{-2/3}+\frac{(2p)^{n}}{T^{1/3-\eta'}}+\left(\frac{\log p}{\log T}\right)^{3/2}\right)\\
&\qquad\leq\frac1{\int_{\Re\nu=0}h_T(\nu)\text{spec}(\nu)\mathrm d \nu}\sum_{\substack{\phi\\
A_\phi(p,p)\in[\alpha,\beta]}}h_T(\nu_\phi)
\end{align*}
for a non-negative weight function $w_2$, which is supported in the interval $[(\alpha+1)/9,(\beta+1)/9]$, identically one in $[(\alpha+\delta+1)/9,(\beta-\delta+1)/9]$, and satisfies $\|w_2'\|_\infty\ll 1/\delta$.

Now it is enough to observe that by (\ref{explicit-formula}), 
\begin{align*}
\int_{T_0/W}\left(1_{[\alpha,\beta]}\left(S_{1,1}(\sigma)\right)-w_i\left(\widetilde S_{1,1}(\sigma)\right)\right)\,\mathrm d \mu_p(\sigma)\ll\delta
\end{align*}
uniformly on $p$ for $i\in\{1,2\}$, and to optimise the choice of $\delta$ and $n$. Choose $\delta\asymp n^{-1/5}$ and $n=\large\lfloor \frac{W\left(5 T^{5/3-5\eta'}\log(2p)\right)}{8A\log (2p)}\large\rfloor$, where $W$ is the Lambert $W$-function. These choices, using the estimate $W(x)\asymp \log x$, yield
\begin{align*}
\delta\asymp\frac1{n^{1/3}\delta^{2/3}}\asymp\frac{(2p)^{n}}{T^{1/3-\eta'}}\asymp\left(\frac{\log p}{\log T}\right)^{1/5}
\end{align*}
and so we have shown that
\begin{align*}
\frac1{\int_{\Re\nu=0}h_T(\nu)\text{spec}(\nu)\mathrm d \nu}\sum_{\substack{\phi\\
A_\phi(p,p)\in[\alpha,\beta]}}h_T(\nu_\phi)=\int_{T_0/W}1_{[\alpha,\beta]}(S_{1,1}(\sigma))\,\mathrm d \mu_p(\sigma)+O\left(\left(\frac{\log p}{\log T}\right)^{1/5}\right),
\end{align*}
as desired. 
\end{proof}

\begin{remark}
The exponent $1/5$ in the error term is probably not optimal. Conjecturally it should be one (compare to the theorems of Lau-Wang and Murty-Sinha). We have not tried to optimise the exponent and used direct methods (i.e. Bernstein polynomials) in order to get a result, which is sufficient for the applications given below.  
\end{remark}

\begin{remark}
Actually just an upper bound for the sum
\begin{align*}
\frac1{\int_{\Re\nu=0}h_T(\nu)\text{spec}(\nu)\mathrm d \nu}\sum_{\substack{\phi\\
A_\phi(p,p)\in[\alpha,\beta]}}h_T(\nu_\phi)
\end{align*}
would suffice for our purposes. We have written down an asymptotic formula as it might have applications in other situations. 
\end{remark}

\noindent We are now ready to prove Theorem \ref{aa}. It follows from \cite[Corollary 2]{Xiao-Xu} that for almost all forms $\phi\in\widetilde{\mathcal H}_T$ there exists a natural number $m$ so that $A_\phi(m,m)<0$ and so by Lemma \ref{positprop} it suffices to show that for all but $\varepsilon\#\widetilde{\mathcal H}_T$ forms $\phi\in\widetilde{\mathcal H}_T$ we have that $A_\phi(m,m)\neq 0 $ for a positive proportion of integers $m$. By \cite[Th\'eor\`eme 14]{Serre} it is enough to show that 
\begin{align*}
\sum_{\substack{p\\    
A_\phi(p,p)=0}}\frac1p<\infty
\end{align*}
for all but $\varepsilon\#\widetilde{\mathcal H}_T$ $\phi$ of $\widetilde{\mathcal H}_T$. We will do this by employing our version of effective Sato-Tate theorem. It is just enough to note that by Theorem \ref{Sato-Tate} we have, for any $\delta>0$,
\begin{align*}
\sum_{\substack{\phi\in\widetilde{\mathcal H}_T}}\sum_{\substack{p\leq X\\
A_\phi(p,p)=0}}\frac1p &\ll\sum_{p\leq X}\frac1p\sum_{\substack{\phi\\
|A_\phi(p,p)|\leq p^{-\delta}}}h_T(\nu_\phi) \\
&\ll\left(\sum_\phi h_T(\nu_\phi)\right) \sum_{p\leq X}\frac1p\left(\frac1{p^\delta}+\left(\frac{\log p}{\log T}\right)^{1/5}\right)\\
&\ll \#\widetilde{\mathcal H}_T
\end{align*}
for any $X>0$ such that, say, $\log X\ll(\log T)^{1/5}$ provided that $T>0$ is sufficiently large. Here the first step follows from the fact that $h_T(\nu_\phi)\gg 1$ for $\phi\in\widetilde{\mathcal H}_T$. This implies that there is an absolute constant $C>0$ such that, for all but at most $\varepsilon\#\widetilde{\mathcal H}_T$ forms in $\widetilde{\mathcal H}_T$, we have
\begin{align*}
\sum_{\substack{p\leq X\\  
A_\phi(p,p)=0}}\frac 1p\leq\frac C{2\varepsilon}.
\end{align*} 
Letting $T\longrightarrow\infty$ gives the desired result if $X$ is chosen so that $X\longrightarrow\infty$ along with $T$. \qed

\section{Proofs of Theorems \ref{non-vanishing-theorem} and \ref{positive-negative}}

It follows from the multiplicativity of Hecke eigenvalues that
\begin{align}\label{non-vanishing}
\#\{m\leq X\,:\,A(m,1)\neq 0\}\asymp X\prod_{\substack{p\leq X\\
A(p,1)\neq 0}}\left(1-\frac 1p\right)
\end{align}
assuming that $A(p,1)<0$ for a positive proportion of primes $p$ \cite[Lemma 2.4.]{Matomaki-Radziwill1} and so we shall show the latter statement for almost all forms. For this we will follow the approach in \cite{Lester-Matomaki-Radziwill}.

In the self-dual case note that as we assume the Ramanujan-Petersson conjecture for $\mathrm{GL}_2$ Maass cusp forms, it follows that $A(p,1)\in[-1,3]$. Hence, the Hecke relations imply 
\begin{align*}
\sum_{\substack{p\leq X\\
A(p,1)<0}}\frac1p &\geq\sum_{p\leq X}\frac{A(p,1)^2-3A(p,1)}{4p}\\
&\geq\sum_{\log X\leq p\leq X^{1/10000}}\frac{1+A(p,p)-3A(p,1)}{4p}\\
&=\frac{1+o(1)}4\log\log X+\sum_{\log X\leq p\leq X^{1/10000}}\frac{A(p,p)-3A(p,1)}{4p}.
\end{align*}
Dividing the sum on the right-hand side into dyadic intervals it follows immediately from Lemma \ref{self-dual-large-sieve} that the sum is $o(\log\log X)$ for almost all self-dual forms if $X\longrightarrow\infty$ as $T\longrightarrow\infty$, concluding the proof in this case.  

For the coefficients $A(p,p)$ we note that $A(p,p)\in[-1,8]$. This follows from the identity $A(p,p)=|A(p,1)|^2-1$ at primes $p$ and the fact that we assume the Ramanujan-Petersson conjecture for $\mathrm{GL}_3$ Maass cusp forms. Using the Hecke relations we have
\begin{align*}
A(p,p)^2-8A(p,p)
=1-6A(p,p)+A(p^2,p^2)+A(p^3,1)+A(1,p^3).
\end{align*}
This yields
\begin{align*}
\sum_{\substack{p\leq X\\
A(p,p)<0}}\frac1p &\geq\sum_{p\leq X}\frac{A(p,p)^2-8A(p,p)}{9p}\\
&\geq\sum_{\log X\leq p\leq X^{1/10000}}\frac{1-6A(p,p)+A(p^2,p^2)+A(p^3,1)+A(1,p^3)}{9p}\\
&=\frac{1+o(1)}9\log\log X+\sum_{\log X\leq p\leq X^{1/10000}}\frac{-6A(p,p)+A(p^2,p^2)+A(p^3,1)+A(1,p^3)}{9p}.
\end{align*}
Again, by dividing the sum on the right-hand side into dyadic intervals and using Lemma \ref{large-sieve} we see that it is $o(\log\log X)$ for almost all forms as $X\longrightarrow\infty$ as $T\longrightarrow\infty$. This completes the proof of Theorem \ref{non-vanishing-theorem}. \qed

Finally, we prove Theorem \ref{positive-negative}. It follows from \cite[Lemma 2.4.]{Matomaki-Radziwill1} that asymptotically half of the non-zero coefficients $A(m,1)$ are positive (and half are negative) if there exists functions $K,L:\mathbb R_+\longrightarrow\mathbb R_+$ so that $K(X)\longrightarrow 0$ and $L(X)\longrightarrow\infty$ as $X\longrightarrow\infty$, and for which
\begin{align*}
\sum_{\substack{p\geq X\\
A(p,1)=0}}\frac1p\leq K(X)\qquad\text{and}\qquad\sum_{\substack{p\leq X\\
A(p,1)<0}}\frac1p\geq L(X)
\end{align*} 
for every $X\geq 2$. 

Above we have shown that 
\begin{align*}
\sum_{\substack{p\\
A(p,1)=0}}\frac 1p<\infty
\end{align*}
for all but $\varepsilon\#\widetilde{\mathcal H}_T$ forms in $\widetilde{\mathcal H}_T$ as $T\longrightarrow\infty$ and that
\begin{align*}
\sum_{\substack{p\leq X\\
A(p,1)<0}}\frac1p\geq\frac{1+o(1)}4\log\log X+o(\log\log X)
\end{align*}
for almost all forms in $\widetilde{\mathcal H}_T$ when $X\longrightarrow\infty$ along with $T$ assuming the Ramanujan-Petersson conjecture for classical Maass cusp forms. Therefore the desired conclusion follows immediately. The second part for the coefficient $A(m,m)$ follows analogously (assuming the Ramanujan-Petersson conjecture for $\mathrm{GL}_3$ Maass cusp forms) from the corresponding results established above. This concludes the proof. \qed

\section{Acknowledgements}

The author wishes to thank Kaisa Matom\"aki for helpful discussions and for an idea, which led to an improvement on the exponent in part (i) of Theorem \ref{merkkivaihtelu} compared to the previous draft. The author would also like to thank Stephen Lester and Abhishek Saha for useful comments and discussions. This work was supported by the Finnish Cultural Foundation and the Engineering and Physical Sciences Research Council [grant number EP/T028343/1].

\end{document}